\documentclass[a4paper,11pt]{amsart}

\usepackage[utf8]{inputenc}
\usepackage[T1]{fontenc}
\usepackage{amssymb,mathenv,amsmath,amsfonts,amsthm}
\usepackage{pgf,tikz}
\usetikzlibrary{arrows}
\usepackage{graphicx}
\usepackage[foot]{amsaddr}
\usepackage{hyperref}

\usepackage{enumerate}
\usepackage[left=3.45cm,right=3.45cm,top=3cm,bottom=3cm]{geometry}
\usepackage{tikz-cd}
\newtheorem{theorem}{Theorem}[section]
\newtheorem{corollary}[theorem]{Corollary}
\newtheorem{lemma}[theorem]{Lemma}

\newtheorem{prop-def}[theorem]{Proposition-Definition}
\newtheorem{prop-defs}[theorem]{Proposition-Definitions}

\newtheorem{proposition}[theorem]{Proposition}

\newtheorem{assumption}[theorem]{Assumption}

\theoremstyle{prooflemma}

\theoremstyle{proofproposition}

\theoremstyle{definition}
\newtheorem{definition}[theorem]{Definition}

\theoremstyle{definitions}

\theoremstyle{remark}
\newtheorem{remark}[theorem]{Remark}

\theoremstyle{condition}
\newtheorem{condition}[theorem]{Condition}

\begin{document}

\title{Local-global compatibility of mod $p$ Langlands program for certain Shimura varieties}

\author{Kegang Liu}

\address{Institut Galilée, Universit\'e Sorbonne
Paris Nord, 99 avenue J.B. Clément, 93430, Villetaneuse, France}
\email{kegang.liu@math.univ-paris13.fr}

\begin{abstract}
We generalize the local-global compatibility result in \cite{Sch18} to higher dimensional cases, by examining the relation between Scholze's functor and cohomology of Kottwitz-Harris-Taylor type Shimura varieties. Along the way we prove a cuspidality criterion from type theory. We also deal with compatibility for torsion classes in the case of semisimple mod $p$ Galois representations with distinct irreducible components under certain flatness hypotheses.

\end{abstract}

\maketitle

\tableofcontents

\newpage

\section{Introduction}
The existence of a $p$-adic local Langlands correspondence, as was first envisioned by Breuil (cf. \cite{Bre10}), is still widely open beyond the case of $\mathrm{GL}_2(\mathbb{Q}_p)$. The work of Caraiani-Emerton-Gee-Geraghty-Paskunas-Shin \cite{CEG$^+$18} provides a construction which associates $p$-adic $\mathrm{GL}_n(L)$-representations with  Galois representations for a $p$-adic field $L$. Their construction is global in nature. Later in \cite{Sch18}, Scholze takes the other direction and constructs Galois representations from mod $p$ (and $p$-adic) representations of $\mathrm{GL}_n(L)$, in a purely local way. Then Scholze proves the compatibility between his construction and the patching construction of Caraiani-Emerton-Gee-Geraghty-Paskunas-Shin, as well as a local-global compatibility result, both for $\mathrm{GL}_2$.

The purpose of this article is to generalize the local-global compatibility result of Scholze to $\mathrm{GL}_n$ for $n>2$. For this we follow mostly the strategy of \cite{Sch18}. Let us describe the results of \cite{Sch18} and this paper in more detail. 

Let $n \geq 1$ be an integer and $L/\mathbb{Q}_p$ be a finite extension with residue field $k$ of cardinality $q$. Denote by $\breve{L}$ the completion of the maximal unramified extension of $L$. Then one has the Lubin-Tate tower $(\mathcal{M}_{\mathrm{LT},J})_{J\subseteq \mathrm{GL}_n(L)}$, indexed by compact open subgroups $J$ of $\mathrm{GL}_n(L)$, consisting of smooth rigid-analytic varieties $\mathcal{M}_{\mathrm{LT},J}$ over $\breve{L}$ equipped with compatible actions of $D^\times$ on all $\mathcal{M}_{\mathrm{LT},J}$ where $D$ is the central division algebra over $L$ of invariant $1/n$. Let $\pi$ be an admissible smooth $\mathbb{F}_p$-representation of $\mathrm{GL}_n(L)$. The construction of Scholze involves descending the trivial sheaf $\pi$ on $\mathcal{M}_{\mathrm{LT},\infty}$ along the Gross-Hopkins period map
$$\pi_{\mathrm{GH}}: \mathcal{M}_{\mathrm{LT},\infty} \to \mathbb{P}^{n-1}_{\breve{L}},$$
resulting in a Weil-equivariant sheaf $\mathcal{F}_\pi$ on the site $(\mathbb{P}_{\breve{L}}^{n-1}/D^\times)_{\text{\'et}}$. (One may refer to Section 3 of \cite{Sch18} for the notations and more details.) The main theorem of \cite{Sch18}, Theorem 1.1 in \textit{loc.cit.}, asserts that for each $i \geq 1$, the cohomology group $H^i_{\text{\'et}}(\mathbb{P}_{\mathbb{C}_p}^{n-1}, \mathcal{F}_\pi)$ is an admissible representation of $D^\times$ and carries an action of the Galois group $\mathrm{Gal}(\bar{L}/L)$; moreover, this cohomology group vanishes if $i > 2(n-1)$. 

To state our local-global compatibility result, we need to change to a global context. So fix a CM field extension $K/F$ with $K$ totally imaginary and $F$ its maximal totally real subfield. Choose a place $\mathfrak{p}$ of $F$ lying over $p$ and an infinite place $\alpha$ of $K$. Let $B$ be a division algebra over $K$ of dimension $n^2$ with an involution of the second kind which is supposed to be positive. Assume that $\mathfrak{p}$ is split in $K$ and fix a place $\mathfrak{q}$ of $K$ over $\mathfrak{p}$ where we assume moreover that $B$ is a division algebra of invariant $1/n$. Then to these data one has a corresponding unitary similitude group $\tilde{G}$ over $F$ with $$\tilde{G}(F_\mathfrak{p})\cong (B_\mathfrak{q}^{op})^\times \times (F_\mathfrak{p})^\times,$$ and for each compact open subgroup $U \subseteq (B_\mathfrak{q}^{op})^\times$ there is a Shimura variety $\mathrm{Sh}_{UC^\mathfrak{p}}$ over $K$ associated with the subgroup $U\times \mathcal{O}_{F_\mathfrak{p}}^\times \times C^\mathfrak{p} $ of $ \tilde{G}(\mathbb{A}_{F,f}) $, where $C^\mathfrak{p} \subseteq \tilde{G}(\mathbb{A}_{F,f}^\mathfrak{p})$ is a fixed sufficiently small tame level. Moreover there exists another division algebra $D$ over $K$ with interchanged local behaviour at $\mathfrak{q}$ and $\alpha$ from $B$, so that its associated unitary similitude group $G'$ over $F$ is an inner form of $\tilde{G}$, locally isomorphic to $\tilde{G}$ at all places except $\mathfrak{p}$ and $\alpha|_F$; in particular, $D$ is split at $\mathfrak{q}$. The space of continuous functions 
$$\pi=\pi_{C^\mathfrak{p}}:=C^0(G'(F)\backslash G'(\mathbb{A}_{F,f})/(\mathcal{O}^\times_{F_\mathfrak{p}}  \times C^{\mathfrak{p}} ), \mathbb{Q}_p/\mathbb{Z}_p)$$
is an admissible $\mathbb{Z}_p$-representation of $\mathrm{GL}_n(F_\mathfrak{p})$ and applying Scholze's functor, one obtains a $(\mathrm{Gal}_{F_{\mathfrak{p}}}\times B_\mathfrak{q}^\times)$-representation on the space $H^{n-1}_{\text{ét}}(\mathbb{P}_{\mathbb{C}_p}^{n-1}, \mathcal{F}_{\pi})$. On the global side, we consider the cohomology of the above system of Shimura varieties and define
$$\rho=\rho_{C^{\mathfrak{p}}}:=\underset{U}{\underrightarrow{\text{lim}}} \, H^{n-1}(\mathrm{Sh}_{UC^{\mathfrak{p}},\overline{K}}, \mathbb{Q}_p/\mathbb{Z}_p)$$
which is a $(\mathrm{Gal}_{K}\times B_\mathfrak{q}^\times)$-representation. The first result that we will prove in this article is the following weak form of local-global compatibility. 
\begin{theorem}
There is a natural isomorphism of $(\mathrm{Gal}_{F_{\mathfrak{p}}}\times B_\mathfrak{q}^\times)$-representations over $\mathbb{Z}_{p}$
$$H^i_{\text{ét}}(\mathbb{P}_{\mathbb{C}_p}^{n-1}, \mathcal{F}_{\pi_{C^{\mathfrak{p}}}}) \cong \rho_{C^{\mathfrak{p}}}.$$
\end{theorem}
As in \cite{Sch18}, this will be proved as a consequence of $p$-adic uniformization of Shimura varieties. We can deduce from this theorem a more precise result using the formalism of $\sigma$-typicity in Section 5 of \cite{Sch18}. To state it, let $\mathbb{T}$ be the formal Hecke algebra over $\mathbb{Z}$ generated by Hecke operators at good places of $K$ and $\mathfrak{m}$ be a maximal ideal of $\mathbb{T}$ (associated with a mod $p$ Galois representation $\bar{\sigma}$) such that 
$$H^{i}(\mathrm{Sh}_{UC^\mathfrak{p},\mathbb{C}},\mathbb{Z}_p)_\mathfrak{m} \neq 0$$
only when $i = n-1$; we also assume that $\mathfrak{m}$ satisfies certain ``strongly irreducibile'' condition (Assumption \ref{tateconj}), to apply known cases of the Tate conjecture (otherwise we need to pass to Galois semisimplifications of relevant cohomology groups) . Then there is an $n$-dimensional Galois representation 
$$\sigma=\sigma_{\mathfrak{m}}: G_{K} \to \mathrm{GL}_{n}(\mathbb{T}(C^\mathfrak{p})_{\mathfrak{m}})$$
characterized by certain Eichler-Shimura
relations. Here $\mathbb{T}(C^\mathfrak{p})_{\mathfrak{m}}$ denotes the completed Hecke algebra at level $C^\mathfrak{p}$ and is a complete local Noetherian ring acting faithfully on $\pi_\mathfrak{m}$. The next result says that one can recover $\sigma|_{\mathrm{Gal}_{F_{\mathfrak{p}}}}$ from $\pi_\mathfrak{m}$.
\begin{theorem}
There is a canonical $\mathbb{T}(C^{\mathfrak{p}})_{\mathfrak{m}}[\mathrm{Gal}_{F_{\mathfrak{p}}}\times B_\mathfrak{q}^\times]$-equivariant isomorphism
$$H^{n-1}_{\text{ét}}(\mathbb{P}_{\mathbb{C}_p}^{n-1}, \mathcal{F}_{\pi_{C^{\mathfrak{p}},\mathfrak{m}}}) \cong \sigma|_{\mathrm{Gal}_{F_{\mathfrak{p}}}}\otimes_{\mathbb{T}(C^{\mathfrak{p}})_{\mathfrak{m}}}\rho[\sigma].$$
for some faithful $\mathbb{T}(C^{\mathfrak{p}})$-module $\rho[\sigma]$ which carries the trivial $\mathrm{Gal}_{F_{\mathfrak{p}}}$-action. If moreover $\bar{\sigma}|_{\mathrm{Gal}_{F_{\mathfrak{p}}}}$ is absolutely irreducible, then this determines $\sigma|_{\mathrm{Gal}_{F_{\mathfrak{p}}}}$ uniquely.
\end{theorem}
One important step in proving this theorem is the following cuspidality criterion, and its consequence on constructing congurences between automorphic forms, which will allow us to extend the Hecke action of $\mathbb{T}$ on $\pi_\mathfrak{m}$ to an action of 
$\mathbb{T}(C^\mathfrak{p})_{\mathfrak{m}}$; see section \ref{digression} for the notations and some background unexplained here. We remark that Fintzen-Shin \cite{FS20} have proved, independently and simultaneously, such results for all reductive groups over totally real fields that are compact modulo center at infinity under a mild condition on $p$; see Theorem 3.1.1 of their paper (and also the Appendix D therein which removes the condition on $p$).
\begin{proposition}
Let $L$ be a $p$-adic field and $\psi: L \to \mathbb{C}^\times$ be a non-trivial (additive) character of level one (that is, $\psi$ is trivial on $\varpi\mathcal{O}_L$ but non-trivial on $\mathcal{O}_L$). Let $\alpha_m$ be the homomorphism $$\alpha_m: U^{\tilde{M}}(\mathfrak{A})\to \varpi^{-N}\mathcal{O}_L,\quad a\mapsto tr_{A/L}(\beta_m(a-1)).$$
If $\pi$ is a smooth irreducible representation of $\mathrm{GL}_n(L)$ such that $\pi|_{U^{\tilde{M}}(\mathfrak{A})}$ contains the character $\psi \circ \alpha_m$, then $\pi$ is cuspidal.
\end{proposition}

\begin{corollary}
Let $A_m=\mathbb{Z}_p[T]/((T^{p^m}-1)/(T-1))$. Take $L=F_\mathfrak{p}$ and let $\psi$ be a character of $L$ with coefficients in $A_m$ whose restriction to $\varpi^{-N}\mathcal{O}_L$ is the map
$$ \varpi^{-N}\mathcal{O}_L\stackrel{\times\varpi^N}{\xrightarrow{\sim}} \mathcal{O}_L \twoheadrightarrow \mathcal{O}_L/\varpi^{me}\twoheadrightarrow \mathbb{Z}/p^m\mathbb{Z} \to A_m^\times$$
with the last arrow mapping $1\in \mathbb{Z}/p^m\mathbb{Z}$ to $T\in A_m^\times$. Define $\psi_m=\psi\circ\alpha_m$. Then any automorphic representation $\pi$ of $G'$ appearing in 
$$C^0(G'(F)\backslash G'(\mathbb{A}_{F,f})/(U^{\tilde{M}}\times \mathcal{O}^\times_{F_\mathfrak{p}}  \times C^{\mathfrak{p}} ), \psi_m)[1/p]$$ 
is cuspidal at $\mathfrak{p}$.
\end{corollary}

Finally we prove a torsion class version under a reasonable flatness assumption, cf. Remark \ref{Gee-Newton} and Remark \ref{rmkflatness} for more on this assumption. It shows in particular that if $\bar{\sigma}|_{\mathrm{Gal}_{F_{\mathfrak{p}}}}$ is irreducible, then it can be read off from $H^{n-1}_{\text{ét}}(\mathbb{P}_{\mathbb{C}_p}^{n-1}, \mathcal{F}_{ \pi[\mathfrak{m}]})$ .

\begin{theorem}\label{modplgc}
Assume that $\pi_{\mathfrak{m}}^\vee$ is flat over $\mathbb{T}(C^{\mathfrak{p}})_{\mathfrak{m}}$. Then $H^{n-1}_{\text{ét}}(\mathbb{P}_{\mathbb{C}_p}^{n-1}, \mathcal{F}_{ \pi[\mathfrak{m}]})$ is a non-zero admissible $\mathrm{Gal}_{F_{\mathfrak{p}}} \times B_\mathfrak{q}^\times $-representation, and any of its indecomposable  $\mathrm{Gal}_{F_{\mathfrak{p}}}$-subrepresentations is a subrepresentation of $\bar{\sigma}|_{\mathrm{Gal}_{F_{\mathfrak{p}}}}$. In particular, if $\bar{\sigma}|_{\mathrm{Gal}_{F_{\mathfrak{p}}}}$ is irreducible, then every indecomposable subrepresentation of $H^{n-1}_{\text{ét}}(\mathbb{P}_{\mathbb{C}_p}^{n-1}, \mathcal{F}_{ \pi[\mathfrak{m}]})$ is isomorphic to $\bar{\sigma}|_{\mathrm{Gal}_{F_{\mathfrak{p}}}}$. 
\end{theorem}

\begin{remark}
In \cite{LLHMPQ21}, Le-Le Hung-Morra-Park-Qian have obtained similar mod $p$ local-global compatibility result in the Fontaine-Laffaille cases under suitable conditions, by studying moduli stacks of Fontaine-Laffaille modules. Moreover, Zicheng Qian has found an argument which can be used to deal with the much larger class of multiplicity-free Galois representations $\bar{\sigma}|_{\mathrm{Gal}_{F_{\mathfrak{p}}}}$ for the local-global compatibility result in Theorem \ref{modplgc} above, cf. \cite{LQ21}.
\end{remark}

The organization is as follows. In section 2 we establish the global setup concerning KHT type Shimura varieties and then describe $p$-adic uniformization of these varieties in section 3. In section 4 we review the construction of Scholze; in section 5 we use the previous results to first obtain a weak version of local-global compatibility. In section 6 we prove a new cuspidality criterion generalizing the $n=2$ case. Finally in section 7 we prove the full local-global compatibility as well as the above-mentioned  result for torsion classes.

\medskip

\textbf{Acknowledgements}. The author would like to thank Pascal Boyer and Stefano Morra, for suggesting this project and for helpful discussions. It is also his pleasure to thank Zicheng Qian, Vincent S\'echerre and Zhixiang Wu for several helpful discussions and for their feedbacks, and thank the anonymous referee for many helpful comments and suggestions. This article is part of the author's PhD thesis and he wishes to thank ED Galil\'ee of Universit\'e Paris 13 and also the ANR grant ANR-14-CE25-0002-01 PerCoLaTor for their support.\\

\newpage

\begin{section}{Global setup }
In this section we consider some unitary similitude Shimura varieties and their cohomology. These so-called Kottwitz-Harris-Taylor type Shimura varieties are singled out and studied in \cite{Kot92}, and then further in \cite{HT01} (on a sub-class of it), hence the terminology. The class of varieties that we concern is a variant of the KHT class, but we still call it KHT type. We will work under the context of \cite{BZ99}, where $p$-adic uniformization results of unitary Shimura varieties are proved which will be used in our work later. Let us first introduce some notations and establish the global setup. One can refer to chapter $0$ of \cite{BZ99} for more details.

Fix an integer $n>2$ and a rational prime integer $p$. Let $K = FE_0$ be a CM field with $F \subseteq K$ totally real and $E_0$ imaginary quadratic. Let $B$ be a division algebra over $K$ of dimension $n^2$ with an involution of the second kind $'$ which we assume to be positive, i.e., for all nonzero $x \in B$ we have $tr_{B/\mathbb{Q}}(xx')>0$. Let $W=B$ as a $B\otimes_KB^{op}$-module and $\Phi: W \times W \to \mathbb{Q}$ an alternating nondegenerate pairing such that 
$$\Phi(bw_1,w_2)=\Phi(w_1,b'w_2)$$
for all $w_1, w_2 \in W$ and $b\in B$.

Let $*$ be the unique involution on $B$ such that
$$\Phi(w_1b,w_2)=\Phi(w_1,w_2b^*)$$
for all $w_1, w_2 \in W$ and $b\in B$. Fix an infinite place $\alpha: K \hookrightarrow \mathbb{C}$ of $K$ and its complex conjugate $\bar{\alpha}$, as well as an embedding $v: \overline{\mathbb{Q}} \hookrightarrow \overline{\mathbb{Q}}_p$.
We denote by $\mathfrak{p}_0, \mathfrak{p}_1,\ldots,\mathfrak{p}_m$ the prime ideals of $\mathcal{O}_F$ lying over $p$ with $\mathfrak{p}_0$ induced by the embedding $v$ above and set $\mathfrak{p}:=\mathfrak{p}_0$. We assume that they all split in $K$ with
$$\mathfrak{p}_i\mathcal{O}_K=\mathfrak{q}_i\bar{\mathfrak{q}}_i,\,\,\,\,\,\,\,\mathfrak{q}_i \neq \bar{\mathfrak{q}}_i,\,\,\,\,\,\,\, i=0,1,\ldots,m.$$
The prime ideals $\mathfrak{q}_0$ and $\bar{\mathfrak{q}}_0$ will also be denoted by $\mathfrak{q}$ and $\bar{\mathfrak{q}}$ respectively. We assume that $B_{\mathfrak{q}}:= B \otimes_K K_{\mathfrak{q}}$ is a division algebra of invariant $1/n$. The  existence  of  such  $B$  and  $\Phi$  follows  from  the  results  of  Kottwitz  and  Clozel (see \cite{Clo91}, Prop. 2.3]). In fact, our choice of $B$ is the same as the one in \cite{HT01} except that the local ramification at the finite place $\mathfrak{q}$ of $K$ where one is interested in could be different.

Let $\tilde{G}$ be the algebraic group over $F$ whose group of $R$-points for any $F$-algebra $R$ is given by
$$\tilde{G}(R):=\{(g,\lambda)\in(B^{op}\otimes_{F}R)^\times\times R^\times |\,gg^*=\lambda\}$$
and let $G:=\text{Res}_{F/\mathbb{Q}}(\tilde{G})$ be the Weil restriction of scalars so that for any $\mathbb{Q}$-algebra $T$ we have
$$G(T):=\{(g,\lambda)\in(B^{op}\otimes_{\mathbb{Q}}T)^\times\times (F\otimes_{\mathbb{Q}}T)^\times |\,gg^*=\lambda\}.$$
We will consider Shimura varieties associated with compact open subgroups $C$ of $$G(\mathbb{A}_f)=\tilde{G}(\mathbb{A}_{F,f})$$ with the form $C=C_pC^p$ where $C^p \subseteq G(\mathbb{A}_f^p)$ is compact open and $$C_p \subseteq G(\mathbb{Q}_p) =\prod_{i=0}^m\tilde{G}(F_{\mathfrak{p}_i})$$ decomposes as 
$$C_p= \prod_{i=0}^mC_{\mathfrak{p}_i},\,\,\,\,\,\,\,\,C_{\mathfrak{p}_i} \subseteq \tilde{G}(F_{\mathfrak{p}_i}).$$
As the involution $*$ induces an isomorphism 
$$B_{\bar{\mathfrak{q}}_i} \xrightarrow{\sim} B_{\mathfrak{q}_i}^{op}, $$
there are identifications for all $i$
$$\tilde{G}(F_{\mathfrak{p}_i}) \cong (B_{\mathfrak{q}_i}^{op})^\times \times F_{\mathfrak{p}_i}^\times.$$
In fact, in later parts $C$ will usually be of the form $C=C_\mathfrak{p}C^\mathfrak{p}$ with $C^\mathfrak{p} \subseteq \tilde{G}(\mathbb{A}_{F,f}^\mathfrak{p})$ (sufficiently small) compact open and $C_\mathfrak{p}=U\times \mathcal{O}^\times_{F_\mathfrak{p}} \subseteq \tilde{G}(F_\mathfrak{p})=(B_\mathfrak{q}^{op})^\times \times F_\mathfrak{p}^\times$ with $U \subseteq (B_\mathfrak{q}^{op})^\times$ compact open, in which case we write $\mathrm{Sh}_{UC^\mathfrak{p}}$ instead of $\mathrm{Sh}_{U\times \mathcal{O}^\times_{F_\mathfrak{p}}   \times C^\mathfrak{p}}$ for the Shimura variety associated with $C$.

Let $\mathbb{S}:=\text{Res}_{\mathbb{C}/\mathbb{R}}(\mathbb{G}_m)$ be the Deligne torus and $h$ be a morphism
$$h: \mathbb{S} \to G_{\mathbb{R}}$$
such that $h$ defines on $W_\mathbb{R}$ a Hodge structure of type $(1,0), (0,1)$ and such that  $\Phi(w_1,h(i)w_2)$ is a symmetric positive definite bilinear form on $W_\mathbb{R}$. Note that $h$ is unique up to $G(\mathbb{R)}$-conjugacy and we let $X$ denote the $G(\mathbb{R)}$-conjugacy class of $h$. Then $(G,X)$ defines a Shimura datum and for sufficiently small compact open subgroups $C \subseteq G(\mathbb{A}_f)$ as above we have a projective system of Shimura varieties $\mathrm{Sh}_C$ over its reflex field, denoted by $E$. 

The morphism $h$ defines a Hodge structure 
$$W_{\mathbb{C}} = W^{1,0} \oplus W^{0,1}$$
and we require the following condition to hold: when the trace of the action of an element $b\in B$ on $W^{1,0}$ is expressed in the form
$$tr_{\mathbb{C}}(b|W^{1,0})= \Sigma_{i: K \to \mathbb{C}}r_i(tr^0b)$$ 
the pair of integers $(r_i, r_{\overline{i}})$ is equal to $(1,n-1)$ at $i=\alpha$ and $(0,n)$ at other infinite places. (Here $tr^0$ denotes reduced trace.) This condition is equivalent to a signature condition as in (2.) of Lemma 1.7.1 of \cite{HT01}, where it is shown that such condition can always be satisfied.
Then we have $E=\alpha(K)$ and $E_v \cong K_{\mathfrak{q}}$. Denote by $\kappa$ the residue field of $E_v$.

We fix a tame level, i.e. a compact open subgroup $C^\mathfrak{p}$ of $\tilde{G}(\mathbb{A}_{F,f}^\mathfrak{p})$ and let $\mathcal{P}$ denote the set of finite places $w$ of $K$ such that
\begin{itemize}
    \item $w|_{\mathbb{Q}} \neq p$;
    \item $w$ is split over $F$;
    \item $B$ is split at $w$ (i.e., $\tilde{G}(F_u) \cong \mathrm{GL}_n(F_u) \times F_u^\times $ where $u=w|_F$) and the component $C_u$ of $C^\mathfrak{p}$ at $u$ is maximal.
\end{itemize}
Consider the abstract Hecke algebra 
$$\mathbb{T}=\mathbb{T}_{\mathcal{P}}:= \mathbb{Z}[T^{(j)}_{w}: w \in \mathcal{P}, j=1,2,\ldots,n]$$
where $T^{(j)}_{w}$ is the Hecke operator corresponding to the double coset 
$$ \Big[\mathrm{GL}_{n}(\mathcal{O}_{F_{w}})
\begin{pmatrix}
\varpi_{w}1_{j} & 0\\
0 & 1_{n-j}
\end{pmatrix}\mathrm{GL}_{n}(\mathcal{O}_{F_{w}}) \Big].$$
Here $\varpi_{w}$ is a uniformizer of the local field $F_w$.
Then the Hecke algebra $\mathbb{T}$ acts on $H^i(\mathrm{Sh}_{UC^\mathfrak{p},\mathbb{C}},\mathbb{Z}_p)$ for all compact open $U\subseteq (B_\mathfrak{q}^{op})^\times$. Fix a finite field $\mathbb{F}_q$ with $q$ elements for $q$ a power of $p$. Let 
$$\bar{\sigma}: \mathrm{Gal}_{K} \to \mathrm{GL}_n(\mathbb{F}_q)$$
be an absolutely irreducible continuous representation of $\mathrm{Gal}_{K}$. We can associate to $\bar{\sigma}$ (together with $\mathcal{P}$) a maximal ideal $\mathfrak{m}$ of $\mathbb{T}$; it is the kernel of the map
$$ \mathbb{T} \to \mathbb{F}_q , \,\,\,\,\,\, T^{(j)}_{w} \mapsto (-1)^j(\textbf{N}w)^{-j(j-1)/2}a^{(j)}_w, \,\,\,\,\,\,  w \in \mathcal{P}, \,\, j=1,2,\ldots,n $$ where $\text{Frob}_w$ for $w \in \mathcal{P}$ is the geometric Frobenius of $\mathrm{Gal}_{K_w}$, $\textbf{N}w$ is the cardinality of the residue field of $K_w$, and the $a^{(j)}_w \in \mathbb{F}_q$ are such that the characteristic polynomial of $\bar{\sigma}(\text{Frob}_w)$ equals
$$X^{n} + \cdots + a^{(j)}_wX^{n-j} + \cdots + a^{(n)}_w.$$
Thus $\mathfrak{m}$ is a maximal ideal containing $p$; we assume further that 

\begin{condition}
For all compact open $U$, we have $$H^{i}(\mathrm{Sh}_{UC^\mathfrak{p},\mathbb{C}},\mathbb{Z}_p)_\mathfrak{m} \neq 0$$
only when $i = n-1$. 
\end{condition}

\begin{remark}
The above condition on vanishing of cohomology outside of the middle degree is satisfied, for instance, when we impose certain generic condition on $\mathfrak{m}$ or on its associated Galois representation, cf. eg. Th\'eorème 4.7 of \cite{Boy17}, and also the main results of \cite{CS15}.
\end{remark}

\begin{remark}
The reason that we impose the above condition is twofold: on the one hand, it gives an easier comparison between the completed cohomology and cohomology with infinite level at $\mathfrak{p}$ in $p$-torsion coefficients, of the Shimura  varieties; on the other hand, under the flatness assumption in Section \ref{seclgc} it implies an important injectivity result (cf. Lemma \ref{injective}).
\end{remark}

Let $\mathbb{T}(UC^\mathfrak{p})$ be the image of $\mathbb{T}$ in $\text{End}(H^{n-1}(\mathrm{Sh}_{UC^\mathfrak{p},\mathbb{C}},\mathbb{Z}))$ and $\mathbb{T}(UC^\mathfrak{p})_\mathfrak{m}$ its $\mathfrak{m}$-adic completion. Note then that $\mathbb{T}(UC^\mathfrak{p})_\mathfrak{m}$ is isomorphic to the localization at $\mathfrak{m}$ of the image of $\mathbb{Z}_p[T^{(i)}_{w}: w \in \mathcal{P}, i=1,2,\ldots,n]$ in $\text{End}(H^{n-1}(\mathrm{Sh}_{UC^\mathfrak{p},\mathbb{C}},\mathbb{Z}_p))$, which one meets more often in the literature. (This is because the localization at a maximal ideal $\mathfrak{m}$ of a finite $\mathbb{Z}_p$-algebra is automatically complete with respect to the  $\mathfrak{m}$-adic topology.) Thus $\mathbb{T}(UC^\mathfrak{p})_\mathfrak{m}$ acts faithfully on $H^{n-1}(\mathrm{Sh}_{UC^\mathfrak{p},\mathbb{C}},\mathbb{Z}_p)_\mathfrak{m}$ and we have an associated Galois representation:

\begin{proposition}\label{assgal}
Assume  from now on that $\mathfrak{m}$ satisfies the following Assumption \ref{tateconj}. There is a unique (up to conjugation) continuous $n$-dimensional Galois representation
$$\sigma = \sigma_{\mathfrak{m}}: \mathrm{Gal}_{K} \to \mathrm{GL}_{n}(\mathbb{T}(UC^\mathfrak{p})_\mathfrak{m})$$
unramified at almost all places and at every place in $\mathcal{P}$, such that for every $w \in \mathcal{P}$, $\sigma(\mathrm{Frob}_{w})$ has characteristic polynomial
$$X^{n} + \cdots + (-1)^{j}(\textbf{N}w)^{j(j-1)/2}T_{w}^{(j)}X^{n-j} + \cdots + (-1)^{n}(\textbf{N}w)^{n(n-1)/2}T_{w}^{(n)}.$$
\end{proposition}

\begin{proof}
As the $\mathrm{Sh}_{UC^\mathfrak{p}}$'s are projective, Matsushima's formula (cf. VII. 5.2 of \cite{BW80}) gives an isomorphism 
$$H^{n-1}(\mathrm{Sh}_{UC^\mathfrak{p}}(\mathbb{C}),\overline{\mathbb{Q}}_p) \cong \bigoplus_{\pi}\pi_f^{U\times \mathcal{O}^\times_{F_\mathfrak{p}}   \times C^\mathfrak{p}} \otimes H^{n-1}(\text{Lie}\,G(\mathbb{R}), K_\infty, \pi_\infty)$$ 
where $\pi$ runs  over irreducible constituents (taken with its multiplicity) of the space of automorphic forms on $G(\mathbb{Q})\backslash G(\mathbb{A})$, and $K_\infty$ is a maximal compact subgroup of $G(\mathbb{R})$. Then by Artin comparison theorem (cf. \cite{SGA4}, Expos\'e XI, Th\'eor\`eme 4.4) we have
$H^{n-1}(\mathrm{Sh}_{UC^\mathfrak{p}}(\mathbb{C}),\overline{\mathbb{Q}}_p) \cong H^{n-1}_{\text{\'et}}(\mathrm{Sh}_{UC^\mathfrak{p},\overline{K}},\mathbb{Z}_p)\otimes_{\mathbb{Z}_p}\overline{\mathbb{Q}}_p$
and we may write 
$$H^{n-1}_{\text{\'et}}(\mathrm{Sh}_{UC^\mathfrak{p},\overline{K}},\overline{\mathbb{Q}}_p) = \bigoplus_{\pi} \pi_f^{U\times \mathcal{O}^\times_{F_\mathfrak{p}}   \times C^\mathfrak{p}} \otimes R^{n-1}(\pi)$$
where $\pi$ runs over cuspidal automorphic representations of $G(\mathbb{A})$ over $\overline{\mathbb{Q}}_p$ (taken with its multiplicity) and where $R^{n-1}(\pi)$ is a finite dimensional continuous representation of $\mathrm{Gal}(\overline{K}/K)$. Localizing both sides at $\mathfrak{m}$, we obtain 
\begin{equation} \label{langlands}
H^{n-1}_{\text{\'et}}(\mathrm{Sh}_{UC^\mathfrak{p},\overline{K}},\overline{\mathbb{Q}}_p)_\mathfrak{m} = \bigoplus_{\pi \in \mathcal{A}(\mathfrak{m})} \pi_f^{U\times \mathcal{O}^\times_{F_\mathfrak{p}}   \times C^\mathfrak{p}} \otimes R^{n-1}(\pi)
\end{equation}
where $\mathcal{A}(\mathfrak{m})$ denotes the subset of cuspidal automorphic representations $\pi$ of $G$ such that the kernel of the corresponding map $\psi_\pi: \mathbb{T}(UC^\mathfrak{p}) \to \overline{\mathbb{Q}}_p$ induced by $\pi$ is contained in $\mathfrak{m}$. As $\bar{\sigma}_\mathfrak{m}$ is assumed to be absolutely irreducible, each of the $\pi$ belonging to $\mathcal{A}(\mathfrak{m})$ will have base change $(\Psi, \Pi)$ to $\mathbb{A}_{E_0}^\times \times \mathrm{GL}_n(\mathbb{A}_K)$ so that $\Pi$ is cuspidal. We now make the following assumption on $\mathfrak{m}$ to apply the known cases of the Tate conjecture in the setting of Shimura varieties, see Remark \ref{remtate} however.

\begin{assumption}\label{tateconj}
Assume that for every $\pi \in \mathcal{A}(\mathfrak{m})$, the Galois representation $\rho_{\Pi,p}$ associated with $\Pi$ by the global Langlands correspondence is strongly irreducible, i.e., $\rho_{\Pi,p}$ is irreducible and not induced from a proper open subgroup of $\mathrm{Gal}_K$.
\end{assumption}
 Thus by Theorem 2.25 of \cite{FN19}, each $R^{n-1}(\pi)$ is a semisimple $\overline{\mathbb{Q}}_p[\mathrm{Gal}(\overline{K}/K)]$-module.
Then by Proposition VII.1.8 and Proposition VI.2.7 of \cite{HT01} (or similarly by Theorem 6.4 and Corollary 6.5 of \cite{Shin11}), for every $\pi$ in $\mathcal{A}(\mathfrak{m})$, $R^{n-1}(\pi)$ is a direct sum of finite copies of an $n$-dimensional $\mathrm{Gal}_K$ representation $\tilde{R}^{n-1}(\pi)$.

We get from \eqref{langlands} an isomorphism
$$\mathbb{T}(UC^\mathfrak{p})_\mathfrak{m} \otimes_{\mathbb{Z}_p} \overline{\mathbb{Q}}_p \xrightarrow{\sim} \prod_{\pi \in \mathcal{A}(\mathfrak{m})} \overline{\mathbb{Q}}_p$$
by sending $T_w^{(i)}$ on the left hand side to its corresponding Hecke eigenvalue on $\pi_w^{\mathrm{GL}_n(\mathcal{O}_w)}$. This isomorphism is defined over a finite extension of $\mathbb{Q}_p$ as $\mathbb{T}(UC^\mathfrak{p})_\mathfrak{m}$ is of finite type over $\mathbb{Z}_p$.
Collecting the $\tilde{R}^{n-1}(\pi)$'s above we obtain a representation
$$\mathrm{Gal}_K \to \mathrm{GL}_{n}(\mathbb{T}(UC^\mathfrak{p})_\mathfrak{m}[1/p]).$$
On the other hand, all characteristic polynomials of Frobenius elements take values in $\mathbb{T}(UC^\mathfrak{p})_\mathfrak{m}$, so one obtains a determinant with values in $\mathbb{T}(UC^\mathfrak{p})_\mathfrak{m}$. But $\bar{\sigma}$ is absolutely irreducible by assumption, so one obtains the existence of the desired representation $\sigma$, cf. Theorem 2.22 of \cite{Che11}. That for every $w \in \mathcal{P}$ the Frobenius action $\sigma(\mathrm{Frob}_{w})$ has the exhibited characteristic polynomial is classical and is proved in Theorem 1 of \cite{Kot92}, cf. also \cite{Wed00}. Finally the uniqueness follows from the fact that the set $\{\mathrm{Frob}_w| w \in \mathcal{P}\}$ is dense inside $\mathrm{Gal}_K$, which is a consequence of Chebotarev's theorem.
\end{proof}

\begin{remark}\label{remtate}
By a strong form of the Tate Conjecture (see eg. \cite{Moo19}), the semisimplicity of the $R^{n-1}(\pi)$'s above always holds true but at this moment it is not fully proved yet, so we made the assumption \ref{tateconj} on $\mathfrak{m}$ to use the results of \cite{FN19}.
\end{remark}

\begin{remark}
In a recent work \cite{Lee22}, S-Y. Lee also obtained similar semisimplicity result for some abelian-type Shimura varieties, building on the work of \cite{FN19}, which however requires again strongly irreducibility hypothesis as above.
\end{remark}

In particular, we have $\sigma \, (\text{mod}\,\mathfrak{m})=\,\overline{\sigma}$. Recall (Definition 5.2, \cite{Sch18}) that for a Noetherian local ring $(R, \mathfrak{m}_R)$, a group $G$ and an $n$-dimensional representation $\sigma_R: G \to \mathrm{GL}_n(R)$ of $G$ whose reduction modulo $\mathfrak{m}_R$ is absolutely irreducible, an $R[G]$-module is said to be $\sigma_R$-typic if there exists an $R$-module $M_0$ such that 
$M=\sigma_R\otimes_R M_0$ with $G$ acting trivially on $M_0$. 
The following proposition generalizes the case of $n=2$.

\begin{proposition} \label{shintypica}
The $\mathbb{T}(UC^\mathfrak{p})_\mathfrak{m}[\mathrm{Gal}_K]$-module $H^{n-1}(\mathrm{Sh}_{UC^\mathfrak{p},\overline{K}},\mathbb{Z}_p)_\mathfrak{m}$ is $\sigma$-typic.
\end{proposition}

\begin{proof}
Since $H^{n-1}(\mathrm{Sh}_{UC^\mathfrak{p},\overline{K}},\mathbb{Z}_p)_\mathfrak{m}$ is a submodule of $H^{n-1}(\mathrm{Sh}_{UC^\mathfrak{p},\overline{K}},\mathbb{Z}_p)_\mathfrak{m}[1/p]$ (over the ring $\mathbb{T}(UC^\mathfrak{p})_\mathfrak{m}[\mathrm{Gal}_K]$), by Proposition 5.4 of \cite{Sch18} it suffices to show that 
$$H^{n-1}(\mathrm{Sh}_{UC^\mathfrak{p},\overline{K}},\mathbb{Z}_p)_\mathfrak{m}\otimes_{\mathbb{Z}_p}\overline{\mathbb{Q}}_p$$ is $\sigma$-typic. But this follows from the description of the $\overline{\mathbb{Q}}_p$-cohomology of $\mathrm{Sh}_{UC^\mathfrak{p},\overline{K}}$, as in the proof of Proposition \ref{assgal}. Indeed, we may shift the multiplicity of each Galois representation $\tilde{R}^{n-1}(\pi)$ to $\pi_f^{U\times \mathcal{O}^\times_{F_\mathfrak{p}}   \times C^\mathfrak{p}}$ and then use the isomorphism of $\mathrm{Gal}_K$-modules ${\pi_f^{U\times \mathcal{O}^\times_{F_\mathfrak{p}}   \times C^\mathfrak{p}}\otimes_{\overline{\mathbb{Q}}_p}\tilde{R}^{n-1}(\pi)} \cong {\pi_f^{U\times \mathcal{O}^\times_{F_\mathfrak{p}}   \times C^\mathfrak{p}} \otimes_{\mathbb{T}(UC^\mathfrak{p})_\mathfrak{m}} (\mathbb{T}(UC^\mathfrak{p})_\mathfrak{m})^{\oplus n}}$ obtained from the following diagram

\begin{center}
\begin{tikzcd}
{\pi_f^{U\times \mathcal{O}^\times_{F_\mathfrak{p}}   \times C^\mathfrak{p}}\otimes_{\overline{\mathbb{Q}}_p}\tilde{R}^{n-1}(\pi)} \arrow[rd] &   & {\pi_f^{U\times \mathcal{O}^\times_{F_\mathfrak{p}}   \times C^\mathfrak{p}} \otimes_{\mathbb{T}(UC^\mathfrak{p})_\mathfrak{m}} (\mathbb{T}(UC^\mathfrak{p})_\mathfrak{m})^{\oplus n}} \arrow[ld] \\
             & {\pi_f^{U\times \mathcal{O}^\times_{F_\mathfrak{p}}   \times C^\mathfrak{p}} \otimes_{\mathcal{T}} (\mathcal{T})^{\oplus n}} &             
\end{tikzcd}
\end{center}
where $\mathcal{T}=\mathbb{T}(UC^\mathfrak{p})_\mathfrak{m} \otimes_{\mathbb{Z}_p}\overline{\mathbb{Q}}_p$ and both arrows are isomorphisms of $\mathrm{Gal}_K$-modules induced by the natural inclusion on the second factor of tensor product. By summing up over $\pi \in \mathcal{A}(\mathfrak{m})$, we obtain an isomorphism
$$\bigoplus_{\pi\in\mathcal{A}(\mathfrak{m})} \big(\pi_f^{U\times \mathcal{O}^\times_{F_\mathfrak{p}}   \times C^\mathfrak{p}}\otimes_{\overline{\mathbb{Q}}_p}R^{n-1}(\pi)\big) \cong  M\otimes_{\mathbb{T}(UC^\mathfrak{p})_\mathfrak{m}}\mathbb{T}(UC^\mathfrak{p})_\mathfrak{m}^{\oplus n}$$ as $\mathrm{Gal}_K$-modules for some $\mathbb{T}(UC^\mathfrak{p})_\mathfrak{m}$-module $M$.

\end{proof}

Now we pass to completed cohomology. Let
$$\tilde{H}^{n-1}(C^\mathfrak{p}, \mathbb{Z}_{p}):=\underset{k}{\underleftarrow{\text{lim}}}\,\underset{U}{\underrightarrow{\text{lim}}}H^{n-1}(\mathrm{Sh}_{UC^\mathfrak{p},\overline{K}}, \mathbb{Z}/p^k\mathbb{Z}),$$
and
$$\tilde{H}^{n-1}(C^\mathfrak{p}, \mathbb{Z}_{p})_\mathfrak{m}:=\underset{k}{\underleftarrow{\text{lim}}}\,\underset{U}{\underrightarrow{\text{lim}}}H^{n-1}(\mathrm{Sh}_{UC^\mathfrak{p},\overline{K}}, \mathbb{Z}/p^k\mathbb{Z})_\mathfrak{m}$$
where $U$ runs over all compact open subgroups of $(B_\mathfrak{q}^{op})^\times$ in each definition.

Then the inverse limit 
$$\mathbb{T}(C^\mathfrak{p})_{\mathfrak{m}}:=\underset{U}{\underleftarrow{\text{lim}}}\mathbb{T}(UC^\mathfrak{p})_{\mathfrak{m}}$$
acts faithfully and continuously on $\tilde{H}^{n-1}(C^\mathfrak{p}, \mathbb{Z}_{p})_\mathfrak{m}$. Then the same argument as in the proof of Proposition 5.7, \cite{Sch18} shows the following result.

\begin{proposition}
There is a unique (up to conjugation) continuous $n$-dimensional Galois representation
$$\sigma=\sigma_{\mathfrak{m}}: G_{K} \to \mathrm{GL}_{n}(\mathbb{T}(C^\mathfrak{p})_{\mathfrak{m}})$$
unramified at almost all places, such that for every $w \in \mathcal{P}$, $\sigma(\text{Frob}_{w})$ has characteristic polynomial
$$X^{n} + \cdots + (-1)^{j}(\textbf{N}w)^{j(j-1)/2}T_{w}^{(j)}X^{n-j} + \cdots + (-1)^{n}(\textbf{N}w)^{n(n-1)/2}T_{w}^{(n)}.$$
The ring $\mathbb{T}(C^\mathfrak{p})_{\mathfrak{m}}$ is a complete Noetherian local ring with finite residue field.
\end{proposition}
\begin{proposition}{\label{typic}}
The $\mathbb{T}(C^\mathfrak{p})_{\mathfrak{m}}[\mathrm{Gal}_{K}]$-module $\tilde{H}^{n-1}(C^\mathfrak{p}, \mathbb{Z}_{p})_\mathfrak{m}$ is $\sigma$-typic.
\end{proposition}
\begin{proof}
This follows from Propostion \ref{shintypica}, noting that the $\sigma$'s are compatible with each other and that all operations in the definition of 
$$\tilde{H}^{n-1}(C^\mathfrak{p}, \mathbb{Z}_{p})_\mathfrak{m}=\underset{k}{\underleftarrow{\text{lim}}}\,\underset{U}{\underrightarrow{\text{lim}}}H^{n-1}(\mathrm{Sh}_{UC^\mathfrak{p},\overline{K}}, \mathbb{Z}/p^k\mathbb{Z})_\mathfrak{m}$$
preserve the property (of modules) of being $\sigma$-typic.
\end{proof}

\end{section}

\begin{section}{$p$-adic uniformization of KHT type Shimura varieties}
Now we describe a $p$-adic (in contrast to classical complex-analytic) uniformization result of the Shimura varieties defined before to compare local and global cohomology groups. We will use the results of Boutot-Zink which introduces a variant of the moduli methods of Rapoport-Zink and which obtains the same uniformization results (up to some Galois twists) with Varshavsky \cite{ Var98I}, \cite{ Var98}. 

We keep the notations from the last section on the global context. Let us consider the following moduli problem, which is a variant of the one in \cite{RZ96}. Recall first that $E$ is the reflex field of the Shimura varieties defined in Section 1 and $v$ is a fixed embedding $v: \overline{\mathbb{Q}} \hookrightarrow \overline{\mathbb{Q}}_p$ which induces a finite place of $E$. For $C$ a compact open subgroup of $\tilde{G}(\mathbb{A}_{F,f})$, let $\mathcal{A}_C$ be the functor on the category of $\mathcal{O}_{E_v}$-schemes, which sends an $\mathcal{O}_{E_v}$-scheme $S$ to the set of isomorphism classes of tuples $(A, \Lambda, \{\lambda_i\}_{i=1}^m, \bar{\eta}^p, \{ \bar{\eta}_{\mathfrak{q}_i}\}_{i=1}^m)$ where

\begin{itemize}
    \item[1.] $A$ is an abelian $\mathcal{O}_{K}$-scheme over $S$ up to isogeny of order prime to $\mathfrak{p}$ together with an action of $\mathcal{O}_{B}$
    $$\iota: \mathcal{O}_{B} \to \mathrm{End}(A);$$
    \item[2.] $\Lambda$ is the one dimensional vector space over $F$ generated by an $F$-homogeneous polarization $\lambda$ of $A$ which is principal in $\mathfrak{p}$;
    \item[3.] for each $i=1, \ldots, m$, $\lambda_i$ is a generator of $\Lambda \otimes_F F_{\mathfrak{p}_i}$ mod ($C_{\mathfrak{p}_i} \cap F_{\mathfrak{p}_i}^\times$);
    \item[4.] $\bar{\eta}^p$ is a class of isomorphisms of $B \otimes \mathbb{A}_f^p$-modules
    $$\bar{\eta}^p: V^p(A) \to W\otimes \mathbb{A}_f^p \,\,\, \mathrm{mod}\, \, C^p$$
    which preserve the Riemannian form on $V^p(A)$ induced by any polarization $\lambda \in \Lambda$ and the pairing $\psi$ on $W\otimes \mathbb{A}_f^p$ up to a constant in $(F \otimes \mathbb{A}_f^p)^\times$;
    \item[5.] for each $i=1, \ldots, m$, $\bar{\eta}_{\mathfrak{q}_i}$ is class of isomorphisms of $B_{\mathfrak{q}_i}$-modules
    $$\bar{\eta}_{\mathfrak{q}_i}: V_{\mathfrak{q}_i}(A) \to W_{\mathfrak{q}_i} \,\,\, \mathrm{mod} \,\, C_{\mathfrak{q}_i}.$$
    
\end{itemize}
such that the following properties are satisfied:

\begin{itemize}
    \item[(a)] The involution $b \to b'$ on $\mathcal{O}_B$ coincides with the one obtained from the Rosati involution on $\mathrm{End}(A)$ induced by $\Lambda$;
    \item[(b)] there is an equality of polynomial functions on $\mathcal{O}_B$, called the determinant condition:
    $$\mathrm{det}_{\mathcal{O}_S}(b, \mathrm{Lie}A)=\mathrm{det}_{\overline{\mathbb{Q}}_p}(b, W_0)$$
\end{itemize}

\begin{remark}
One can formulate condition (b) in terms of the associated $p$-divisible group $X$ of $A$. Indeed, the determinant condition amounts to requiring: (1) $X_\mathfrak{q}$ is a special formal $\mathcal{O}_{B_\mathfrak{q}}$-module à la Drinfeld; (2) $X_{\mathfrak{q}_i}$ is \'etale for each $i=1, \ldots, m$.
\end{remark}

It is proved in \cite{BZ99} that
\begin{proposition}

If $C$ is sufficiently small, then the \'etale sheafification of $\mathcal{A}_C$ is represented by a projective scheme over $\mathcal{O}_{E_v}$, still denoted by $\mathcal{A}_C$. For such varying $C$ these schemes form a projective system $\{\mathcal{A}_C\}$ with finite transition maps.

\end{proposition}

It is natural to equip the system $\{\mathcal{A}_C\}$ with a right action of $G(\mathbb{A}_f)$, for the details one may refer to \cite{BZ99}. Next we consider moduli problems for $p$-divisible formal $\mathcal{O}_{B_\mathfrak{q}}$-modules. We fix a special formal $\mathcal{O}_{B_\mathfrak{q}}$-module $\Phi$ over $\overline{\kappa}$ and denote its dual by $\hat{\Phi}$. We will often write $\mathbb{X}$ for $\Phi \times \hat{\Phi}$.

Let $\mathcal{M}$ be the functor on the category of $\mathcal{O}_{\breve{E}_v}$-schemes where $p$ is locally nilpotent, which assigns such a scheme $T$ the set of isomorphism classes of pairs $(X, \rho)$ where
\begin{itemize}
    \item[(1)] $X$ is a $p$-divisible $\mathcal{O}_{B_\mathfrak{p}}$-module of special type over $T$;
    \item[(2)] $\rho$ is a quasi-isogeny of $\mathcal{O}_{B_\mathfrak{p}}$-modules over $T$
    $$\rho: \mathbb{X} \times_{\mathrm{Spec}\,\kappa} \overline{T} \to X \times_T \overline{T}.$$ 
\end{itemize}
Here $\overline{T} := T \times_{\mathrm{Spec}\,\mathbb{Z}_p} \mathrm{Spec}\,\mathbb{F}_p$ and it is regarded as a scheme over $\mathrm{Spec}\,\overline{\kappa}$ via $\overline{T} \to \mathrm{Spec}\, (\mathcal{O}_{\breve{E}_v}/p) \to \mathrm{Spec}\,\overline{\kappa}$. Note that by rigidity of quasi-isogenies, $\rho$ gives rise to a quasi-isogeny $\beta: \hat{X_2} \to X_1$; we require that Zariski locally on $T$ one can find an element $h \in F_\mathfrak{p}^\times$ such that $h\beta$ is an isomorphism.

\begin{proposition}
$\mathcal{M}$ is representable by a $p$-adic formal scheme over $\mathcal{O}_{\breve{E}_v}$, which is equipped with a Weil descent datum. Furthermore, it acquires natural actions of two groups, the group $J$ of self quasi-isogenies of the $\mathcal{O}_{B_\mathfrak{p}}$-module $\mathbb{X}$ preserving the polarization up to a constant in $F_\mathfrak{p}^\times$, and $\tilde{G}(F_\mathfrak{p})$.
\end{proposition}

Now we are ready to state the $p$-adic uniformization result. For this we fix a point $(A_s, \Lambda_s, \{\lambda_{s.i}\}_{i=1}^m, \bar{\eta_s}^p, \{ \bar{\eta}_{s,\mathfrak{q}_i}\}_{i=1}^m) \in \mathcal{A}_C(\bar{\kappa})$ for a sufficiently small $C$. Using a theorem of Serre-Tate, one can define a uniformization morphism
$$\Theta: \mathcal{M} \times \tilde{G}(\mathbb{A}_{F,f}^\mathfrak{p})/C^\mathfrak{p} \to \mathcal{A}_C \times_{\mathrm{Spec}\, \mathcal{O}_{E_v}} \mathrm{Spec}\, \mathcal{O}_{\breve{E}_v}$$
which is $\tilde{G}(\mathbb{A}_{F,f})$-equivariant. It is then natural to try to find out the fibers of this map. Define
$$\tilde{I}^\bullet(F) = \{\phi \in  \mathrm{End}_B^0(A_s) \,| \, \phi\circ\phi' \in F^\times \}$$
where $\phi \mapsto \phi'$ is the Rosati involution induced by $\Lambda$ on the finite dimensional $\mathbb{Q}$-algebra $\mathrm{End}_B^0(A_s)$. (Here $\mathrm{End}_B^0(A_s)$ means $\mathrm{End}_B(A_s) \otimes_\mathbb{Z} \mathbb{Q}$ which is usually called the endomorphism ring up to isogeny.) Then we regard $\tilde{I}^\bullet$ as an algebraic group over $F$ such that its $F$-points are given as above.

Let $G'$ be the inner form of $G$ over $F$ such that $G'(F \otimes_{\mathbb{Q}}\mathbb{R})$ is compact modulo center, $G'(\mathbb{A}_{F,f}^\mathfrak{p}) = \tilde{G}(\mathbb{A}_{F,f}^\mathfrak{p})$, and $G'(F_{\mathfrak{p}}) \cong \mathrm{GL}_n(F_\mathfrak{p}) \times \mathcal{O}_{F_\mathfrak{p}}^\times$. Then $G'$ is the unitary similitude group associated with $(D,\mu)$ where 
$D$ is a division algebra over $K$ of dimension $n^2$ with an involution $\mu$ of the second kind over $F$ satisfying the following conditions:
\begin{itemize}
    \item $D$ splits at $\mathfrak{q}$,
    \item $(D, \mu)$ and $(B, ')$ are locally isomorphic at all finite places except $\mathfrak{q}$,
    \item $\mu$ is positive definite at all archimedean places of $F$.
\end{itemize}

\begin{remark}
The  existence  of  such  $D$  and  $\mu$  follows  from  the  results  of  Kottwitz  and  Clozel (see \cite{Clo91}, Prop. 2.3]).
\end{remark}

\begin{theorem}{$($\textbf{$p$-adic uniformization of Shimura varieties}, $\mathrm{Theorem \,\, 0.16}$, \cite{BZ99}$)$}\label{unifor}\\
For any compact open subgroup $C\subseteq \tilde{G}(\mathbb{A}_{F,f})$, there is an isomorphism of rigid analytic spaces over $Sp\breve{E}_v$:
$$I^\bullet (\mathbb{Q}) \backslash (\mathfrak{X} \times F_\mathfrak{p}^\times) \times \tilde{G}(\mathbb{A}_{F,f}^\mathfrak{p})/C \xrightarrow{\sim} \mathrm{Sh}^{rig}_{(G,\tilde{h}),C} \times_{SpE_v}Sp \breve{E}_v$$
which is $\tilde{G}(\mathbb{A}_{F,f})$-equivariant and compatible with the Weil descent data on both sides.
\end{theorem}
We explain a bit on the notations and the proof of this theorem. It is essentially a consequence of the corresponding $p$-adic uniformization of Shimura varieties that are moduli spaces of abelian varieties with PEL structures. Here $\tilde{h}$ is a well chosen morphism $\mathbb{S} \to G_\mathbb{R}$ in the definition of Shimura datum to account for certain Galois twist which occurs when comparing some $\mathrm{Sh}_{(G,h)}$ and its appropriate PEL type avatar. Then one proves the  following uniformization result for PEL type Shimura varieties following the method of Rapoport-Zink, \cite{RZ96}.

\begin{theorem}{$($$\mathrm{Page \,\, 29}$, \cite{BZ99}$)$}\label{pel}
There is an isomorphism of rigid analytic spaces over $\breve{E}_v$:
$$I^\bullet (\mathbb{Q}) \backslash (\mathfrak{X} \times F_\mathfrak{p}^\times) \times \tilde{G}(\mathbb{A}_{F,f}^\mathfrak{p})/C \xrightarrow{\sim} \mathcal{A}_C^{rig} \times_{\mathrm{Sp}\,E_v} \mathrm{Sp}\, \breve{E}_v.$$
\end{theorem}
Here and above, $I^\bullet$ denotes $\text{Res}_{F/\mathbb{Q}}\tilde{I}^\bullet$ and $\tilde{I}^\bullet$ is defined above as the algebraic group of certain self quasi-isogenies of an abelian variety over the residue field of $K_\mathfrak{q}$. It is easy to show that $\tilde{I}^\bullet$ is in fact isomorphic to our $G'$. Thus we can write $I^\bullet(\mathbb{Q}) = \tilde{I}^\bullet(F) = G'(F)$. Moreover, $\mathfrak{X}$ denotes the rigid analytic pro-covering space over $\breve{\mathcal{N}}^{rig}$, the rigid analytic space associated with the formal scheme $\breve{\mathcal{N}}$ classifying quasi-isogenies of some fixed special formal $O_{B_\mathfrak{q}}$-module over $\bar{\kappa}$, see e.g. 5.34 of \cite{RZ96} for more details. With Theorem \ref{pel},  the $p$-adic uniformization of Shimura varieties Theorem \ref{unifor} is proved by a comparison between these two types of Shimura varieties, see Lemma 0.9 of \cite{BZ99}.

Recall that we have fixed a compact open subgroup $C^\mathfrak{p}$ of $\tilde{G}(\mathbb{A}_{F,f}^\mathfrak{p})$, that $U$ denotes a compact open subgroup of 
$(B_\mathfrak{q}^{op})^\times$ and that $\mathrm{Sh}_{UC^\mathfrak{p}}$ denotes the Shimura variety associated with the subgroup $U\times \mathcal{O}_{F_\mathfrak{p}}^\times \times C^\mathfrak{p}$. Thus taking $C=U\times \mathcal{O}_{F_\mathfrak{p}}^\times \times C^\mathfrak{p}$ in Theorem \ref{unifor} (then base change to $\mathbb{C}_p$ and take the associated adic spaces with quasi-separated rigid analytic spaces, cf. 1.1.11 of \cite{Hub96}) we obtain:
\begin{theorem}
There is an isomorphism of adic spaces over $\mathbb{C}_p$ 
 $$(\mathrm{Sh}_{UC^\mathfrak{p}}\otimes_E \mathbb{C}_p)^{ad} \cong G'(F)\backslash \mathcal{M}_{\mathrm{Dr},U, \mathbb{C}_p} \times (F_\mathfrak{p}^\times/\mathcal{O}_{F_\mathfrak{p}}^\times) \times G'(\mathbb{A}_{F,f}^{\mathfrak{p}})/C^\mathfrak{p} $$
 compatible with varying $U \subseteq (B_\mathfrak{q}^{op})^\times$ and with the Weil descent datum to $E_v.$ 
\end{theorem}
In this theorem, $\mathcal{M}_{\mathrm{Dr},U, \mathbb{C}_p}$ denotes the Drinfeld space of level $U \subseteq (B_\mathfrak{q}^{op})^\times$, which is a smooth adic space over $\text{Spa}(\mathbb{C}_p, \mathcal{O}_{\mathbb{C}_p})$, cf. \cite{SW13}. When $U$ varies, these spaces form the so-called Drinfeld tower $(\mathcal{M}_{\mathrm{Dr},U, \mathbb{C}_p})_U$ with finite \'etale transitive maps and they are related to the above pro-space $\mathfrak{X}$ via
$$(\mathfrak{X} \otimes_{\breve{E}_v}\mathbb{C}_p)^{ad}  \, \cong \, \underset{U}{\underleftarrow{\text{lim}}} \mathcal{M}_{\mathrm{Dr},U, \mathbb{C}_p}.$$

\end{section}

\begin{section}{Review of Scholze's construction}
We recall briefly here the definition and basic properties of the functor defined by Scholze which associates Galois representations with admissible mod $p$ (and $p$-adic) representations of $p$-adic Lie groups. One may refer to \cite{Sch18} for more details.

In this section $L$ is a finite extension of $\mathbb{Q}_p$, and $J$ is an open subgroup of $\mathrm{GL}_n(L)$. Denote by $\breve{L}$ the completion of the maximal unramified extension of $L$ and by $\mathcal{M}_{\mathrm{LT},\infty}$ the perfectoid space over $\breve{L}$ constructed in \cite{SW13}, so that 
$$\mathcal{M}_{\mathrm{LT},\infty} \sim \underset{K}{\underleftarrow{\text{lim}}}\,\mathcal{M}_{\mathrm{LT},J}.$$
Here $\mathcal{M}_{\mathrm{LT},J}$ is the smooth rigid-analytic Lubin-Tate space over $\breve{L}$ at finite level $J$, cf. \cite{GH94}. Write $D$ for the moment to be the central division algebra over $L$ of invariant $1/n$.
Then for $\pi$ an admissible $\mathbb{F}_p$-representation of $\mathrm{GL}_n(L)$, one may construct a sheaf $\mathcal{F}_\pi$ on the site $(\mathbb{P}^{n-1}_{\breve{L}}/D^\times)_{\text{\'et}}$ equivariant for the Weil descent datum, by descending the trivial sheaf $\pi$ along the Gross-Hopkins map
$$ \pi_{\mathrm{GH}}: \mathcal{M}_{\mathrm{LT},\infty} \to \mathbb{P}^{n-1}_{\breve{L}} $$
which can be considered as a $\mathrm{GL}_n(L)$-torsor.
\begin{proposition}

The association mapping a $D^\times$-equivariant \'etale map $U \to  \mathbb{P}^{n-1}_{\breve{L}}$ to the $\mathbb{F}_p$-vector space
$$\mathrm{Map}_{\mathrm{cont},\mathrm{GL}_n(L)\times D^\times}(|U\times_{\mathbb{P}_{\breve{L}}^{n-1}}\mathcal{M}_{\mathrm{LT},\infty}|,\pi)$$
of continuous $\mathrm{GL}_n(L)\times D^\times$-equivariant maps defines a Weil-equivariant sheaf $\mathcal{F}_\pi$ on $(\mathbb{P}^{n-1}_{\breve{L}}/D^\times)_{\text{\'et}}$. The association $\pi \mapsto \mathcal{F}_\pi$ is exact, and all geometric fibers of $\mathcal{F}_\pi$ are isomorphic to $\pi$.

\end{proposition}

Let $C/\breve{L}$ be an algebraically closed complete extension with ring of integers $\mathcal{O}_C$. 

\begin{theorem}{$(\mathrm{Scholze})$} 
For any $i\geq 0$, the $D^\times$-representation $H^{i}_{\text{ét}}(\mathbb{P}_{C}^{n-1}, \mathcal{F}_\pi)$ is admissible, independent of $C$, and vanishes for $i > 2(n-1)$. Taking $C=\mathbb{C}_p$, the action of the Weil group $W_L$ on $H^{i}_{\text{ét}}(\mathbb{P}_{\mathbb{C}_p}^{n-1}, \mathcal{F}_\pi)$ extends to an action of the absolute Galois group $G_L$ of $L$. More generally, the same statements hold when $\pi$ is replaced by an admissible $A[\mathrm{GL}_n(L)]$-module with $A$ a complete Noetherian local ring with finite residue field of characteristic $p$. 
\end{theorem}

\end{section}

\begin{section}{The case of $\mathrm{GL}_n$}
Now let us return to our case of $\mathrm{GL}_n, n > 2$.

\begin{definition}
Let $\rho_{C^{\mathfrak{p}}}^i$ be the admissible $\mathbb{Z}_p$-representation of $G_{K}\times (B_\mathfrak{q}^{op})^\times$ given by
$$\rho_{C^{\mathfrak{p}}}^i:=\underset{U}{\underrightarrow{\text{lim}}} \, H^i(\mathrm{Sh}_{UC^{\mathfrak{p}},\overline{K}}, \mathbb{Q}_p/\mathbb{Z}_p)$$
where $U$ runs over compact open subgroups of $(B_\mathfrak{q}^{op})^\times$. We will often write $\rho$ for $\rho_{C^{\mathfrak{p}}}^{n-1}$.
Let $\pi=\pi_{C^{\mathfrak{p}}}$ be the admissible $\mathbb{Z}_p$-representation of $\mathrm{GL}_n(F_{\mathfrak{p}})$ given by the space of continuous functions 
$$\pi=\pi_{C^{\mathfrak{p}}}:=C^0(G'(F)\backslash G'(\mathbb{A}_{F,f})/(\mathcal{O}^\times_{F_\mathfrak{p}}  \times C^{\mathfrak{p}} ), \mathbb{Q}_p/\mathbb{Z}_p).$$
\end{definition} 
Define $$\mathrm{Sh}_{C^{\mathfrak{p}}}:=G'(F)\backslash \mathcal{M}_{\mathrm{Dr},\infty, \mathbb{C}_p} \times G'(\mathbb{A}_{F,f}^{\mathfrak{p}})/C^{\mathfrak{p}}.$$ 
It is a perfectoid space over $\mathbb{C}_p$ equipped with a Weil descent datum to $E_v$, such that we have the following similarity relation between adic spaces (see Def. 2.4.1 of \cite{SW13})
$$\mathrm{Sh}_{C^{\mathfrak{p}}} \sim \underset{U}{\underleftarrow{\text{lim}}}(\mathrm{Sh}_{UC^{\mathfrak{p}}}\otimes_E \mathbb{C}_p)^{ad}.$$
 
In particular, we have

\begin{equation}\label{infcoho}
  H^i(\mathrm{Sh}_{C^{\mathfrak{p}}, \mathbb{C}_p}, \mathbb{Q}_p/\mathbb{Z}_p) = \underset{U}{\underrightarrow{\text{lim}}}\,H^i(\mathrm{Sh}_{UC^{\mathfrak{p}}, \mathbb{C}_p,},\mathbb{Q}_p/\mathbb{Z}_p) 
\end{equation}
as $(W_{F_{\mathfrak{p}}}\times B_{\mathfrak{q}}^\times)$-representations where $W_{F_{\mathfrak{p}}}$ is the Weil group (here we drop the notation ``op'' from $(B_{\mathfrak{q}}^{op})^\times$ so that it acts from the right on $\rho_{C^{\mathfrak{p}}}^i=\underset{U}{\underrightarrow{\text{lim}}}\,H^i(\mathrm{Sh}_{UC^{\mathfrak{p}}, \mathbb{C}_p,},\mathbb{Q}_p/\mathbb{Z}_p)$).

\begin{remark}
It follows from the proper base change theorem (e.g., \href{https://stacks.math.columbia.edu/tag/0A5I}{Tag 0A5I} of \cite{Sta}) that there is a canonical isomorphism
$$H^i(\mathrm{Sh}_{UC^{\mathfrak{p}}, \mathbb{C}_p,},\mathbb{Q}_p/\mathbb{Z}_p) \cong H^i(\mathrm{Sh}_{UC^{\mathfrak{p}}, \overline{K}},\mathbb{Q}_p/\mathbb{Z}_p),$$
which we will use tacitly throughout.
\end{remark}

Then we can proceed as Scholze did in \cite{Sch18} to obtain first a weak form of $p$-adic local-global compatibility. By Proposition  7.1.1 of \cite{SW13} we have the Hodge-Tate period morphism
$$\pi_{\mathrm{HT}}: \mathcal{M}_{\mathrm{Dr},\infty,\mathbb{C}_{p}} \to \mathbb{P}_{\mathbb{C}_p}^{n-1}$$ compatible with Weil descent data, and it is identified with the Grothendieck-Messing period map under the duality isomorphism
$$\mathcal{M}_{\mathrm{Dr},\infty,\mathbb{C}_{p}}\cong \mathcal{M}_{\mathrm{LT},\infty,\mathbb{C}_{p}}\,\, ,$$
cf. Theorem 7.2.3 of \cite{SW13}. Now the $\mathrm{GL}_n(F_{\mathfrak{p}})$-equivariance of the Hodge-Tate period map induces a map
$$\pi_{\mathrm{HT}}^{\mathrm{Sh}}: \mathrm{Sh}_{C^{\mathfrak{p}}, \mathbb{C}_p} = G'(F)\backslash \mathcal{M}_{\mathrm{Dr},\infty, \mathbb{C}_p} \times G'(\mathbb{A}_{F,f}^{\mathfrak{p}})/C^{\mathfrak{p}} \to \mathbb{P}_{\mathbb{C}_p}^{n-1}.$$
Note that $\pi_{\mathrm{HT}}$ is $(W_{F_{\mathfrak{p}}}\times B_\mathfrak{q}^\times)$-equivariant.

\begin{proposition}\label{sheafiso}
There is a $(W_{F_{\mathfrak{p}}}\times B_\mathfrak{q}^\times)$-equivariant isomorphism of sheaves on the étale site of $($the adic space$)$ $\mathbb{P}_{\mathbb{C}_p}^{n-1}:$
$$R\pi_{\mathrm{HT}\text{\'et}*}^{\mathrm{Sh}}(\mathbb{Q}_p/\mathbb{Z}_p) \cong \mathcal{F}_{\pi_{C^\mathfrak{p}}}.$$
\end{proposition}

\begin{proof}
The proof is the same as in \cite{Sch18}, but we write it again here. We first check that the higher direct images vanish. It suffices to check it at stalks, so let $\bar{x} = \mathrm{Spa}(C, C^+) \to \mathbb{P}_{\mathbb{C}_p}^{n-1}$ be a geometric point, that is, $C/\breve{E}_v$ is complete algebraically closed and $C^+\subseteq C$ is an open and bounded valuation subring. We may
assume that $C$ is the completion of the algebraic closure of the residue field of $\mathbb{P}_{\mathbb{C}_p}^{n-1}$ at the image of $\bar{x}$. Let $$\bar{x} \to U_i \to \mathbb{P}_{\mathbb{C}_p}^{n-1}$$ be a cofinal system of \'etale neighborhoods of $\bar{x}$; then we have $\bar{x} \sim \underset{i}{\underleftarrow{\text{lim}}}U_i$.
Let $$U_i^{\mathrm{Sh}} \to \mathrm{Sh}_{C^\mathfrak{p},\mathbb{C}_p}$$
be the pullback of $U_i$, so that $U_i^{\mathrm{Sh}}$ is a perfectoid space \'etale over $\mathrm{Sh}_{C^\mathfrak{p},\mathbb{C}_p}$. One can form the inverse limit $U_x^{\mathrm{Sh}} = \underset{i}{\underleftarrow{\text{lim}}}U_i^{\mathrm{Sh}}$ in the category of perfectoid spaces over $\mathbb{C}_p$. We have equalities 
$$(R^j\pi_{\mathrm{HT}\text{\'et}*}^{\mathrm{Sh}}(\mathbb{Q}_p/\mathbb{Z}_p))_{\bar{x}} = \underset{i}{\underrightarrow{\text{lim}}}H^j_{\text{\'et}}(U_i^{\mathrm{Sh}},\mathbb{Q}_p/\mathbb{Z}_p) = H^j_{\text{\'et}}(U_{\bar{x}}^{\mathrm{Sh}},\mathbb{Q}_p/\mathbb{Z}_p).$$
On the other hand, the fiber $U_{\bar{x}}^{\mathrm{Sh}}$ is given by profinitely many copies of $\bar{x}$,
$$U_{\bar{x}}^{\mathrm{Sh}} = \mathrm{Spa}(C^0(G'(F)\backslash \mathrm{GL}_n(F_\mathfrak{p})\times \tilde{G}(\mathbb{A}_{F,f}^\mathfrak{p})/C^\mathfrak{p},C),C^0(G'(F)\backslash \mathrm{GL}_n(F_\mathfrak{p})\times \tilde{G}(\mathbb{A}_{F,f}^\mathfrak{p})/C^\mathfrak{p},C^+)).$$
This implies that $H^j_{\text{\'et}}(U_{\bar{x}}^{\mathrm{Sh}},\mathbb{Q}_p/\mathbb{Z}_p)$ vanishes for $j > 0$, and equals 
$$C^0(G'(F)\backslash \mathrm{GL}_n(F_\mathfrak{p})\times \tilde{G}(\mathbb{A}_{F,f}^\mathfrak{p})/C^\mathfrak{p},\mathbb{Q}_p/\mathbb{Z}_p)$$
in degree $0$.

It remains to identify $\pi_{\mathrm{HT}\text{\'et}*}^{\mathrm{Sh}}(\mathbb{Q}_p/\mathbb{Z}_p)$. The previous computation shows that the fibers are isomorphic to $\pi_{C^\mathfrak{p}}$. Let $U \to \mathbb{P}_{\mathbb{C}_p}^{n-1}$ be an \'etale map. We need to construct a map 
\begin{equation*}
\begin{aligned}
&H^0(U\times_{\mathbb{P}_{\mathbb{C}_p}^{n-1}}(G'(F)\backslash \mathcal{M}_{\mathrm{Dr},\infty, \mathbb{C}_p} \times G'(\mathbb{A}_{F^,f}^{\mathfrak{p}})/C^\mathfrak{p}), \mathbb{Q}_p/\mathbb{Z}_p) \\
&\to \mathrm{Map}_{\mathrm{Cont},\text{ GL}_n(F_\mathfrak{p})}(|U\times_{\mathbb{P}_{\mathbb{C}_p}^{n-1}}\mathcal{M}_{\mathrm{Dr},\infty, \mathbb{C}_p}|,C^0(G'(F)\backslash \mathrm{GL}_n(F_\mathfrak{p})\times G'(\mathbb{A}_{F,f}^\mathfrak{p})/C^\mathfrak{p},\mathbb{Q}_p/\mathbb{Z}_p)).
\end{aligned}
\end{equation*}
The left hand side is equal to
$$C^0(|U\times_{\mathbb{P}_{\mathbb{C}_p}^{n-1}}(G'(F)\backslash \mathcal{M}_{\mathrm{Dr},\infty, \mathbb{C}_p} \times G'(\mathbb{A}_{F^,f}^{\mathfrak{p}})/C^\mathfrak{p})|, \mathbb{Q}_p/\mathbb{Z}_p),$$
and it remains to observe that
there is a natural $ \mathrm{GL}_n(F_\mathfrak{p})$-equivariant map
\begin{equation*}
\begin{aligned}
(U\times_{\mathbb{P}_{\mathbb{C}_p}^{n-1}}\mathcal{M}_{\mathrm{Dr},\infty, \mathbb{C}_p}) \times G'(F)\backslash \mathrm{GL}_n(F_\mathfrak{p})\times G'(\mathbb{A}_{F,f}^\mathfrak{p})/C^\mathfrak{p}) \\
 \to U\times_{\mathbb{P}_{\mathbb{C}_p}^{n-1}}(G'(F)\backslash \mathcal{M}_{\mathrm{Dr},\infty, \mathbb{C}_p} \times G'(\mathbb{A}_{F^,f}^{\mathfrak{p}})/C^\mathfrak{p}).
\end{aligned}
\end{equation*}
\end{proof}

It follows from Proposition \ref{sheafiso} that $$H^i(\mathrm{Sh}_{C^{\mathfrak{p}}, \mathbb{C}_p}, \mathbb{Q}_p/\mathbb{Z}_p) = H^i_{\text{ét}}(\mathbb{P}_{\mathbb{C}_p}^{n-1}, \mathcal{F}_{\pi_{C^{\mathfrak{p}}}}).$$
Combining this with formula \eqref{infcoho} we have proved
\begin{theorem}\label{cmp}
There is a natural isomorphism of $(\mathrm{Gal}_{F_{\mathfrak{p}}}\times B_\mathfrak{q}^\times)$-representations over $\mathbb{Z}_{p}$
$$H^i_{\text{ét}}(\mathbb{P}_{\mathbb{C}_p}^{n-1}, \mathcal{F}_{\pi_{C^{\mathfrak{p}}}}) \cong \rho_{C^{\mathfrak{p}}}^i.$$
It is clearly $\mathbb{T}$-equivariant.
\end{theorem}

This is a form of local-global compatibility. We will now derive a more concrete version of it by combining it with the $\sigma$-typicity decomposition before, so as to specialize both sides of the above isomorphism to their ``$\mathfrak{m}$-parts''. To this end we need to extend the Hecke action and for this we need a cuspidality criterion.

\end{section}

\begin{section}{Digression on type theory and a cuspidality criterion}\label{digression}

Let $n\geq2$ be an integer as before and $L/\mathbb{Q}_{p}$ a finite extension. Set $G=\mathrm{GL}_{n}(L)$.
Just as the $n=2$ case, for general $n$ we also have a cuspidality criterion of admissible representations of $G$ in terms of containment of certain characters of filtration subgroups of $G$. To achieve this, we will work with a special extreme case in type theory. The reference for the following and more general framework is \cite{BK96}.

Let $E/L$ be the totally ramified extension of degree $n$ obtained by adjoining an $n$-th root $\alpha$ of a fixed uniformizer $\omega$ of $F$, namely $E=L[\alpha]$ with $\alpha^n=\omega$. Regarding $E=:V$ as an $n$-dimensional vector space over $L$, we have isomorphisms $\mathrm{End}_L(E) \cong M_n(L) =:A$ and $\mathrm{Aut}_L(E) \cong \mathrm{GL}_n(L) = G$. Let $\mathfrak{A}$ be the hereditary order corresponding to the chain lattice $\mathfrak{L}=\{\mathfrak{p}_E^i: i\in \mathbb{Z}\}$, that is, $\mathfrak{A} = \mathrm{End}_{\mathfrak{o}_L}^0(\mathfrak{L}):=\{x\in A\,|\, x\mathfrak{p}_E^i\subseteq \mathfrak{p}_E^{i}\,, \forall i\in \mathbb{Z} \}$. Then $\mathfrak{A}$ is the unique $\mathfrak{o}_L$-order in $A$ such that $$E^\times \subseteq \mathfrak{L}(\mathfrak{A}):=\{g\in G\,|\, g\mathfrak{A}g^{-1}=\mathfrak{A}\};$$ in other words, $E^\times$ normalises $\mathfrak{A}$. Moreover, $\mathfrak{A}$ is a principal order. We set $\mathfrak{P}=\mathrm{End}_{\mathfrak{o}_L}^1(\mathfrak{L}):=\{x\in A\,|\, x\mathfrak{p}_E^i\subseteq \mathfrak{p}_E^{i+1}\,, \forall i\in \mathbb{Z} \}$, the Jacobson radical of $\mathfrak{A}$.

Next, for $m$ a positive integer, we can and do choose an element $\beta_m\in E$ with $\nu_E(\beta_m)=-mn-1$ (for example, one can take $\beta_m=\varpi^{-m}\alpha^{-1}$); here $\nu_E$ is the discrete valuation on $E$. Then it follows that $\beta_m$ is minimal over $L$ in the sense of (1.4.14), Page 41 of \cite{BK96}. Recall that we also have the notion of a simple stratum as defined in (1.5.5) Definition, Page 43 of \cite{BK96}.
\begin{lemma}
Set $\beta_m=\varpi^{-m}\alpha^{-1} \in E$. Then the stratum $[\mathfrak{A},-\nu_{\mathfrak{A}}(\beta_m),0,\beta_m]$ is simple. 
\end{lemma}
\begin{proof}
It is clear that $E=L[\beta_m]$, so $L[\beta_m]$ is a field, and its non-zero elements normalise $\mathfrak{A}$ by our very choice of $\mathfrak{A}$ above. On the other hand, $\nu_{\mathfrak{A}}(\beta_m)$ is by definition the unique integer $r$ such that $\beta_m \in \mathfrak{P}^r\backslash\mathfrak{P}^{r+1}$, which is easily computed to be $-(mn+1)$. Thus $-\nu_{\mathfrak{A}}(\beta_m) >0$ and each condition in the definition of simple stratum is satisfied (see also the comments right after (1.5.5) Definition in \cite{BK96}).
\end{proof}
\begin{remark}
Note that $M:=-\nu_{\mathfrak{A}}(\beta_m) = mn+1 >0$ tends to infinity as $m$ tends to infinity. Denote $\tilde{M}:=[\frac{M}{2}]+1 >0$ for future use.
\end{remark}

Recall that for the hereditary order $\mathfrak{A}$ defined above, one has a sequence of filtration subgroups $U^k(\mathfrak{A})$ given by $U^k(\mathfrak{A}):=1+\mathfrak{P}^k$ for $k\geq1$ and $U^0(\mathfrak{A}):=\mathfrak{A}^\times$. Also in our case, $B:=\mathrm{End}_E(V)\cong E, \mathfrak{B}:=\mathfrak{A}\cap B\cong \mathcal{O}_E$ is a hereditary $\mathfrak{o}_E$-order with Jacobson radical $\mathfrak{Q}=\mathfrak{P}\cap B\cong \mathfrak{p}_E$. In particular, since $\beta_m$ is minimal, we have $H^1(\beta_m, \mathfrak{A})= (1+\mathfrak{p}_E)U^{\tilde{M}}(\mathfrak{A})$, cf. Corollary 3.1.13 of \cite{Con10}, hence $$H^1/U^{\tilde{M}}(\mathfrak{A})\cong (1+\mathfrak{p}_E)/(1+\mathfrak{p}_E^{\tilde{M}}).$$In the last step we have used the following two facts: (1) The property that $E^\times$ normalizes $U^{\tilde{M}}$ implies that $H^1$ normalizes $U^{\tilde{M}}$; (2) $H^1\cap U^{\tilde{M}} =(1+\mathfrak{p}_E^{\tilde{M}})$; both follow from easy computations. Moreover, $det_B: E^\times \to E^\times$ is the identity map in our case. In the following proposition, we keep the notations as above. Thus, $M=mn+1$ and we also set $N=M-\tilde{M}$.
\begin{proposition}\label{cuspidal}
Let $\psi: L \to \mathbb{C}^\times$ be a non-trivial (additive) character of level one (that is, $\psi$ is trivial on $\varpi\mathcal{O}_L$ but non-trivial on $\mathcal{O}_L$). Let $\alpha_m$ be the homomorphism $$\alpha_m: U^{\tilde{M}}(\mathfrak{A})\to \varpi^{-N}\mathcal{O}_L,\quad a\mapsto tr_{A/L}(\beta_m(a-1)).$$
If $\pi$ is a smooth irreducible representation of $\mathrm{GL}_n(L)$ such that $\pi|_{U^{\tilde{M}}(\mathfrak{A})}$ contains the character $\psi \circ \alpha_m$, then $\pi$ is cuspidal.
\end{proposition}
\begin{proof}
We will prove that $\pi$ contains a maximal simple type; hence by Theorem 6.2.2 of \cite{BK96} it is cuspidal. Let us first show that $\pi$ contains a simple character $\theta$ of $H^1$. By assumption, there is an injection $$\psi_m:=\psi\circ\alpha_m \hookrightarrow \pi|_{U^{\tilde{M}}(\mathfrak{A})}$$ which gives rises to a non-zero homomorphism $$\text{c-Ind}_{U^{\tilde{M}}(\mathfrak{A})}^{H^1}\psi_m\to \pi|_{H^1}.$$ It then suffices to show that we have the following decomposition: $$\text{c-Ind}_{U^{\tilde{M}}(\mathfrak{A})}^{H^1}\psi_m=\underset{\theta \in C^1(\psi_m)}{ \bigoplus}\theta$$ where $C^1(\psi_m)$ is the set of simple characters of $H^1$ extending $\psi_m$. Clearly the right hand side is contained in the left; thus the equality holds if both sides have the same dimension. But by Mackey's restriction formula, the left side has dimension $$[H^1:U^{\tilde{M}}(\mathfrak{A})]=|(1+\mathfrak{p}_E)/(1+\mathfrak{p}_E^{\tilde{M}})|$$ while by 1. (b), Proposition 3.1.18 of \cite{Con10} the right side has dimension $$[U^1(\mathfrak{B}):U^{\tilde{M}}(\mathfrak{B})]=|(1+\mathfrak{p}_E)/(1+\mathfrak{p}_E^{\tilde{M}})|.$$ So they indeed have the same dimension, as claimed. We have the inclusions of compact open subgroups of $G$: 
$$H^1 \subseteq J^1 \subseteq J$$
Since there is a unique extension $\eta(\theta)$ of the simple character $\theta$ to $J^1$ and $J^1$ is compact, it follows that $\pi|_{J^1}$ contains this unique $\eta(\theta)$. Our final task is to show that $\pi|_{J}$ contains a simple type $(J, \lambda)$. 

Note that as a result of our choice of the hereditary order and the stratum above, one has $ef=n/[E:L]=1$, hence $e=f=1$ in our case. Thus $$J/J^1\cong GL(f, k_E)^e =k_E^\times$$ and the factor $\sigma$ in the definition of $\lambda=\kappa\otimes\sigma$ is nothing but a character inflated from $$k_E^\times \cong \mathfrak{o}_E^\times/(1+\mathfrak{o}_E).$$ So by Theorem 5.2.2 of \cite{BK96} the irreducible representation $\lambda$ in the definition of simple type coincides with the $\beta$-extension $\kappa$ of $\eta(\theta)$. Therefore it suffices to show that $\pi|_J$ contains some $\beta$-extension of $\eta(\theta)$. We make a counting argument again. Indeed, we obtain from the containment $\eta(\theta)\hookrightarrow \pi|_{J^{1}}$ a non-zero map
$$\text{c-Ind}_{J^1}^J\eta(\theta) \to \pi|_J.$$
We claim that $$\text{c-Ind}_{J^1}^J\eta(\theta)=\bigoplus \kappa$$ where $\kappa$ runs over all $\beta$-extensions of $\eta(\theta)$. But again, the left side has dimension by Mackey's formula $$[J:J^1] = |GL(f, k_E)^e|= |k_E^\times|$$ while the right side has dimension  $|\mathfrak{o}_E^\times/(1+\mathfrak{o}_E)|=|k_E^\times|$ by Theorem 5.2.2 of \cite{BK96}. We conclude that $\pi$ contains a simple type which is also maximal (namely $e=e(\mathfrak{B}|\mathfrak{o}_E)=1$) and the proof is complete.
\end{proof}
\begin{remark}
In this proposition, we could have just dealt with such a cuspidality criterion by assuming the containment of a simple character in $\pi$, for which the proof would be a little easier. The reason that we consider characters on the smaller subgroups $U^{\tilde{M}}$ rather than $H^1$'s is that the latter does not form a cofinal system of subgroups of $\mathrm{GL}_{n}(L)$ (that is, they cannot be as small as one might want), while the former does. This point will be important in Corollary \ref{ext}.
\end{remark}

Let $e$ be the ramification index of $[F_{\mathfrak{p}}:\mathbb{Q}_p]$ and fix a surjection $\mathcal{O}_L/\varpi^{me}\twoheadrightarrow \mathbb{Z}/p^m\mathbb{Z}$. Recall that $F$ is global field; now taking $L=F_{\mathfrak{p}}$, we have the following corollary (here we fix an embedding $A_m^\times \to \mathbb{C}^\times$).
\begin{corollary}\label{cuspidality}
Let $A_m=\mathbb{Z}_p[T]/((T^{p^m}-1)/(T-1))$. Let $\psi$ be a character of $L$ with coefficients in $A_m$ whose restriction to $\varpi^{-N}\mathcal{O}_L$ is the map
$$ \varpi^{-N}\mathcal{O}_L\stackrel{\times\varpi^N}{\xrightarrow{\sim}} \mathcal{O}_L \twoheadrightarrow \mathcal{O}_L/\varpi^{me}\twoheadrightarrow \mathbb{Z}/p^m\mathbb{Z} \to A_m^\times$$
with the last arrow mapping $1\in \mathbb{Z}/p^m\mathbb{Z}$ to $T\in A_m^\times$. Define $\psi_m=\psi\circ\alpha_m$. Then any automorphic representation $\pi$ of $G'$ appearing in 
$$C^0(G'(F)\backslash G'(\mathbb{A}_{F,f})/(U^{\tilde{M}}\times \mathcal{O}^\times_{F_\mathfrak{p}}  \times C^{\mathfrak{p}} ), \psi_m)[1/p]$$ 
is cuspidal at $\mathfrak{p}$.
\end{corollary}
\begin{proof}
Let $k$ be the smallest integer such that $\psi$ is trivial on $\varpi^k\mathcal{O}_L$, so $-N<k\leq me-N$. Let $\psi'$ be the character of $L$ given by $x\mapsto \psi(\varpi^{k-1}x)$ so that $\psi'$ is a character of level one. Then $\pi_{\mathfrak{p}}|_{U^{\tilde{M}_{m+k-1}}}$ contains the character $\psi' \circ \alpha_{m+k-1}$, where $\tilde{M}_{m+k-1}=n(m+k-1)+1$; indeed, for any $g\in U^{\tilde{M}_{m+k-1}}$ we have 
$$\pi(g)v=\psi(\alpha_m(g))v=\psi'(\varpi^{-(k-1)}\alpha_m(g))v=\psi'\circ \alpha_{m+k-1}(g)v$$
where $v$ is an eigenvector of $\pi|_{U^{\tilde{M}}}$. Therefore one concludes by the proposition above.
\begin{remark}
The existence of such an additive character $\psi$ of $L$ in the above corollary can be seen from the following fact: for each integer $k$, one can always extend a character $\chi$ on $\omega^{k}\mathcal{O}_L$ to $\omega^{k-1}\mathcal{O}_L$, since $\mathbb{C}^\times$ is an injective $\mathbb{Z}$-module. (Here we are regarding each $\omega^{k}\mathcal{O}_L$ as a $\mathbb{Z}$-module, too.)
\end{remark}
\end{proof}

Let us return to our context. Note that as $G$ and $G'$ are isomorphic to $\mathrm{GL}_1(B) \times_{\mathbb{G}_m}\text{Res}_{K/F}\mathbb{G}_m$ and $\mathrm{GL}_1(D) \times_{\mathbb{G}_m}\text{Res}_{K/F}\mathbb{G}_m$ respectively, applying the results of \cite{LS19} along with the classical Jacquet-Langlands correspondence between $\mathrm{GL}_1(B)$ (resp. $\mathrm{GL}_1(D)$) and $\mathrm{GL}_n$, we know that an automorphic representation of $G'$ transfers to $G$ if and only if it is a discrete series at $\mathfrak{p}$.
\begin{corollary}{\label{ext}}
Let $\mathbb{T}(C^{\mathfrak{p}})_\mathfrak{m}$ be defined as in Section 2, so that it acts faithfully on $H^{n-1}(C^{\mathfrak{p}}, \mathbb{Q}_{p}/\mathbb{Z}_p)_\mathfrak{m}$. The natural action of $\mathbb{T}$ on
$$\pi_\mathfrak{m}=\pi_{{C^{\mathfrak{p}}},\mathfrak{m}}=C^0(G'(F)\backslash G'(\mathbb{A}_{F,f})/(\mathcal{O}^\times_{F_\mathfrak{p}}  \times C^{\mathfrak{p}} ), \mathbb{Q}_p/\mathbb{Z}_p)_\mathfrak{m}$$
extends to a continuous action of $\mathbb{T}(C^{\mathfrak{p}})_\mathfrak{m}$.
\end{corollary}
\begin{proof}
We follow the proof of Scholze in \cite{Sch18}.
It is enough to prove this for each group
$$C^0(G'(F)\backslash G'(\mathbb{A}_{F,f})/(K'\times \mathcal{O}^\times_{F_\mathfrak{p}}  \times C^{\mathfrak{p}} ), \mathbb{Z}/p^m\mathbb{Z})_\mathfrak{m}$$
with $K'\subseteq \mathrm{GL}_n(F_\mathfrak{p})$ compact open subgroups. We may assume that $K' = U^{\tilde{M}}$ is of the form in Proposition \ref{cuspidal} with $\tilde{M} = [\frac{mn+1}{2}]+1 $ for varying $m$, as these subgroups form a cofinal system. In this case, $\mathbb{Z}/p^m\mathbb{Z} \cong A_m/(T-1)$ and $\psi_m$ mod $(T-1)$ is trivial. Hence there is a $\mathbb{T}$-equivariant surjection
\begin{equation*}
\begin{aligned}
 C^0(G'(F)\backslash G'(\mathbb{A}_{F,f})/(U^{\tilde{M}} \times  \mathcal{O}_{F_\mathfrak{p}}^\times & \times  C^{\mathfrak{p}}), \psi_m)_\mathfrak{m} \\
 & \to C^0(G'(F)\backslash G'(\mathbb{A}_{F,f})/(U^{\tilde{M}} \times \mathcal{O}_{F_\mathfrak{p}}^\times \times C^{\mathfrak{p}}), \mathbb{Z}/p^m\mathbb{Z})_\mathfrak{m}.
 \end{aligned}
\end{equation*}
Thus it suffices to show that the action of $\mathbb{T}$ on
$$M:=C^0(G'(F)\backslash G'(\mathbb{A}_{F,f})/(U^{\tilde{M}} \times \mathcal{O}^\times_{F_\mathfrak{p}}  \times C^{\mathfrak{p}} ), \psi_m)_\mathfrak{m}$$
extends continuously to $\mathbb{T}(C^{\mathfrak{p}})_\mathfrak{m}$. But $M$ is $p$-torsion free, so it is enough to check in characteristic $0$. In that case, by Corollary \ref{cuspidality} all automorphic representations of $G'$ appearing in $M[1/p]$ are cuspidal at the place $\mathfrak{p}$ and thus transfer by Jacquet-Langlands correspondence to $\tilde{G}$, so that after transfer they show up in the cohomology group
$H^{n-1}(\mathrm{Sh}_{{U^{\tilde{M}}C^\mathfrak{p}},\mathbb{C}}, \mathbb{Z}_p)_{\mathfrak{m}}$
for $U^{\tilde{M}}$ sufficiently small. Since $\mathbb{T}(C^{\mathfrak{p}})_\mathfrak{m} \twoheadrightarrow \mathbb{T}(U^{\tilde{M}}C^{\mathfrak{p}})_\mathfrak{m}$ acts continuously on $$H^{n-1}(\mathrm{Sh}_{{U^{\tilde{M}}C^\mathfrak{p}},\mathbb{C}}, \mathbb{Z}_p)_{\mathfrak{m}},$$ the result follows.

\end{proof}
\end{section}

\begin{section}{The local-global compatibility}\label{seclgc}

With the previous preparations, let us now deduce the concrete form of local-global compatibility mentioned before. Recall that there is an $n$-dimensional Galois representation
$$\sigma=\sigma_{\mathfrak{m}}: G_{K} \to \mathrm{GL}_{n}(\mathbb{T}(C^{\mathfrak{p}})_{\mathfrak{m}})$$ 
associated with $\mathfrak{m}$, as well as its reduction modulo $\mathfrak{m}$
$$\bar{\sigma}=\bar{\sigma}_{\mathfrak{m}}: G_{K} \to \mathrm{GL}_{n}(\mathbb{F}_q)$$
which is by definition modular; here $q$ is the cardinality of $\mathbb{T}(C^{\mathfrak{p}})/\mathfrak{m}$. By Proposition $\ref{typic}$ and Corollary \ref{ext}, $\rho^{n-1}_{C^{\mathfrak{p}},\mathfrak{m}}$ is a $\sigma$-typic $\mathbb{T}(C^{\mathfrak{p}})_{\mathfrak{m}}[\mathrm{Gal}_K]$-module, so we have
$$\rho^{n-1}_{C^{\mathfrak{p}},\mathfrak{m}}=\sigma\otimes_{\mathbb{T}(C^{\mathfrak{p}})_{\mathfrak{m}}}\rho[\sigma]$$
for some $\mathbb{T}(C^{\mathfrak{p}})_{\mathfrak{m}}[B_\mathfrak{q}^\times]$-module $\rho[\sigma]$. 

\begin{corollary}
There is a canonical $\mathbb{T}(C^{\mathfrak{p}})_{\mathfrak{m}}[\mathrm{Gal}_{F_{\mathfrak{p}}}\times B_\mathfrak{q}^\times]$-equivariant isomorphism
$$H^{n-1}_{\text{ét}}(\mathbb{P}_{\mathbb{C}_p}^{n-1}, \mathcal{F}_{\pi_{C^{\mathfrak{p}},\mathfrak{m}}}) \cong \sigma|_{\mathrm{Gal}_{F_{\mathfrak{p}}}}\otimes_{\mathbb{T}(C^{\mathfrak{p}})_{\mathfrak{m}}}\rho[\sigma].$$
The $\mathbb{T}(C^{\mathfrak{p}})$-module $\rho[\sigma]$ is faithful.
\end{corollary}

\begin{proof}
As taking cohomology commutes with localization on $\mathbb{T}$, we have 
$$H^{n-1}_{\text{ét}}(\mathbb{P}_{\mathbb{C}_p}^{n-1}, \mathcal{F}_{\pi_{C^{\mathfrak{p}},\mathfrak{m}}}) \cong H^{n-1}_{\text{ét}}(\mathbb{P}_{\mathbb{C}_p}^{n-1}, \mathcal{F}_{\pi_{C^{\mathfrak{p}}}})_\mathfrak{m}.$$
By Theorem \ref{cmp} the right hand side is isomorphic to $\rho^{n-1}_{C^{\mathfrak{p}},\mathfrak{m}}$, which by the comments before the corollary admits the desired $\sigma$-typic decomposition. 
\end{proof}

\begin{remark}
In the proof of the above corollary it is guaranteed by corollary \ref{ext} that there is a well defined $\mathbb{T}(C^{\mathfrak{p}})_{\mathfrak{m}}$-module structure on $\pi_\mathfrak{m}$ (so that we can use the formalism of $A[G]$-modules). We also identified tacitly the two groups $\mathrm{Gal}_{K_\mathfrak{q}}$ and $\mathrm{Gal}_{F_\mathfrak{p}}$ under the canonical isomorphism $F_\mathfrak{p} \cong K_\mathfrak{q}$.
\end{remark}
By the formalism of $\sigma$-typicity of Section 5 of \cite{Sch18}, this implies that the localization $\pi_{{C^{\mathfrak{p}},\mathfrak{m}}}$ determines the local Galois representation
$$\sigma|_{\mathrm{Gal}_{F_{\mathfrak{p}}}}: {\mathrm{Gal}_{F_{\mathfrak{p}}}} \to \mathrm{GL}_{n}(\mathbb{T}(C^{\mathfrak{p}})_{\mathfrak{m}})$$
at least when $\bar{\sigma}|_{\mathrm{Gal}_{F_{\mathfrak{p}}}}$ is absolutely irreducible. 

\begin{theorem}
Assume that $\bar{\sigma}|_{\mathrm{Gal}_{F_{\mathfrak{p}}}}$ is absolutely irreducible. Then $\sigma|_{\mathrm{Gal}_{F_{\mathfrak{p}}}}$ is determined by $\pi_{C^\mathfrak{p},\mathfrak{m}}$. To be more precise, the $\mathbb{T}(C^{\mathfrak{p}})_{\mathfrak{m}}[\mathrm{Gal}_{F_{\mathfrak{p}}}]$-module
$$H^{n-1}_{\text{ét}}(\mathbb{P}_{\mathbb{C}_p}^{n-1}, \mathcal{F}_{\pi_{C^{\mathfrak{p}},\mathfrak{m}}})$$
is $\sigma|_{\mathrm{Gal}_{F_{\mathfrak{p}}}}$-typic, and faithful as a $\mathbb{T}(C^{\mathfrak{p}})_\mathfrak{m}$-module; it follows from Lemma 5.5 of \cite{Sch18} that this module determines $\sigma|_{\mathrm{Gal}_{F_{\mathfrak{p}}}}$uniquely.
\end{theorem}
This more concrete form of local-global compatibility shows that the local component at $\mathfrak{p}$ of the global Galois representation $\sigma_\mathfrak{m}$ associated with the eigensystem $\mathfrak{m}$ is determined compatibly by $\pi_{C^\mathfrak{p},\mathfrak{m}}$ through the cohomology group $H^{n-1}_{\text{ét}}(\mathbb{P}_{\mathbb{C}_p}^{n-1}, \mathcal{F}_{\pi_{C^{\mathfrak{p}},\mathfrak{m}}})$. Now we want to prove a similar result for the $\mathfrak{m}$-torsion $\pi_{C^{\mathfrak{p}}}[\mathfrak{m}]$ instead of the localization $\pi_{C^{\mathfrak{p}},\mathfrak{m}}$. 

Consider the following two short exact sequences of $\mathbb{T}(C^{\mathfrak{p}})_{\mathfrak{m}}$-modules
\begin{equation}\label{shortes1}
0 \to \pi_{C^{\mathfrak{p}},\mathfrak{m}}[\mathfrak{m}] \overset{i}{\to} \pi_{C^{\mathfrak{p}},\mathfrak{m}} \to \pi_{C^{\mathfrak{p}},\mathfrak{m}}/ \pi_{C^{\mathfrak{p}},\mathfrak{m}}[\mathfrak{m}] \to 0
\end{equation}
and
\begin{equation}\label{shortes2}
0 \to \pi_{C^{\mathfrak{p}},\mathfrak{m}}/ \pi_{C^{\mathfrak{p}},\mathfrak{m}}[\mathfrak{m}] \overset{j}{\to} \prod _{i=1}^{r}\pi_{C^{\mathfrak{p}},\mathfrak{m}} \to Q \to 0
\end{equation}
where $r$ is an integer such that there exists a sequence of generators $f_1,\ldots,f_r$ of the ideal $\mathfrak{m}\mathbb{T}(C^{\mathfrak{p}})_{\mathfrak{m}}$, $j$ induced by mapping $x\in \pi_{C^{\mathfrak{p}},\mathfrak{m}}$ to $(f_1x,\ldots,f_rx) \in \prod _{i=1}^{r}\pi_{C^{\mathfrak{p}},\mathfrak{m}}$, and $Q$ is by definition the quotient in the second sequence. Let $C$ denote $ \pi_{C^{\mathfrak{p}},\mathfrak{m}}/ \pi_{C^{\mathfrak{p}},\mathfrak{m}}[\mathfrak{m}]$. We have two corresponding long exact sequences
\begin{equation}\label{longes1}
     \ldots \to H^{n-1}_{\text{ét}}(\mathbb{P}_{\mathbb{C}_p}^{n-1}, \mathcal{F}_{\pi_{C^{\mathfrak{p}},\mathfrak{m}}[\mathfrak{m}]}) \overset{i_*}{\to} H^{n-1}_{\text{ét}}(\mathbb{P}_{\mathbb{C}_p}^{n-1}, \mathcal{F}_{\pi_{C^{\mathfrak{p}},\mathfrak{m}}}) \to H^{n-1}_{\text{ét}}(\mathbb{P}_{\mathbb{C}_p}^{n-1}, \mathcal{F}_{C}) \to \ldots
\end{equation}
and
\begin{equation}\label{longes2}
    \ldots \to H^{n-1}_{\text{ét}}(\mathbb{P}_{\mathbb{C}_p}^{n-1}, \mathcal{F}_{C}) \to \bigoplus_{i=1}^{r}H^{n-1}_{\text{ét}}(\mathbb{P}_{\mathbb{C}_p}^{n-1}, \mathcal{F}_{\pi_{C^{\mathfrak{p}},\mathfrak{m}}}) \to H^{n-1}_{\text{ét}}(\mathbb{P}_{\mathbb{C}_p}^{n-1}, \mathcal{F}_{Q}) \to \ldots
\end{equation}

Now for simplicity, $\pi_{C^{\mathfrak{p}}}$ will be denoted by $\pi$.
\begin{lemma}
The natural map $\pi \rightarrow \pi_{\mathfrak{m}}$ induces an isomorphism $\pi[\mathfrak{m}] \xrightarrow{\sim} \pi_{\mathfrak{m}}[\mathfrak{m}]$.
\end{lemma}

\begin{proof}
Suppose $f\in \pi[\mathfrak{m}]$ maps to $0$ in $\pi_\mathfrak{m}$, then there exists $t \notin \mathfrak{m}$ such that $tf=0$. But $\mathfrak{m}$ is a maximal ideal, so there exist $r\in \mathbb{T}$ and $m\in \mathfrak{m}$ satisfying $rt+m=1$, and we deduce that $f= rtf+mf=0$ since $f$ is $\mathfrak{m}$-torsion. To show surjectivity, let $f/s \in \pi_\mathfrak{m}[\mathfrak{m}]$. As $$\pi = \varinjlim_{K_\mathfrak{p}}\varinjlim_{r}C^0(G'(F)\backslash G'(\mathbb{A}_{F,f})/(K_\mathfrak{p}\times \mathcal{O}^\times_{F_\mathfrak{p}}  \times C^{\mathfrak{p}} ), \mathbb{Z}/p^r\mathbb{Z})$$
we may assume that $f$ belongs to a member indexed by $K_\mathfrak{p}$ and $r$ above, which is Hecke stable and finite as a set; denote it by $\pi_{K_\mathfrak{p},r}$. Then the image of $\mathbb{T}$ (resp. $\mathfrak{m}$) in $\text{End}(\pi_{K_\mathfrak{p},r})$ is a finite ring (resp. ideal) and we choose $m_1,\ldots,m_l \in \mathfrak{m}$ so that their images form a set equal to the image of $\mathfrak{m}$. By the assumption that $f/s \in \pi_\mathfrak{m}[\mathfrak{m}]$, there exists an element $t_i \in \mathbb{T}$ with $t_i \notin \mathfrak{m}$ for each $1 \leq i \leq l$ such that $t_im_if = 0$. Let $t=t_1\ldots t_l$, so that $tmf=0$ for every $m\in \mathfrak{m}$. Moreover there exists $h\in \mathbb{T}$ satisfying $1-hst \in \mathfrak{m}$. Taking $g=fht$, one verifies that $g\in \pi[\mathfrak{m}]$ and that $f/s$ is the image of $g$ under the localization map.
\end{proof}

\begin{lemma}\label{flatness}
Assume that $\pi_{\mathfrak{m}}^\vee= \mathrm{Hom}_{\mathbb{Z}_p}(\pi_{\mathfrak{m}}, \mathbb{Q}_p/\mathbb{Z}_p)$ is flat over $\mathbb{T}(C^{\mathfrak{p}})_{\mathfrak{m}}$ (see the remark below). Then for all $r \geq 1$, we have
$$\pi_{\mathfrak{m}}[\mathfrak{m}^{r+1}]/\pi_{\mathfrak{m}}[\mathfrak{m}^r] \cong \bigoplus_s \pi[\mathfrak{m}]$$ where $s = \rm{dim}_{\mathbb{T}/\mathfrak{m}} (\mathfrak{m}^r/\mathfrak{m}^{r+1})$.
\end{lemma}

\begin{proof}
Since taking dual (with respect to $\mathbb{Q}_p/\mathbb{Z}_p$) is an exact functor, we have
\begin{align*}
    (\pi_{\mathfrak{m}}[\mathfrak{m}^{r+1}]/\pi_{\mathfrak{m}}[\mathfrak{m}^r])^\vee \cong & Ker(\pi_{\mathfrak{m}}[\mathfrak{m}^{r+1}]^\vee \to \pi_{\mathfrak{m}}[\mathfrak{m}^{r}]^\vee)  \\ \cong  
    & Ker(\pi_{\mathfrak{m}}^\vee/\mathfrak{m}^{r+1}\pi_{\mathfrak{m}}^\vee \to \pi_{\mathfrak{m}}^\vee/\mathfrak{m}^{r}\pi_{\mathfrak{m}}^\vee) \\
    \cong & \mathfrak{m}^{r}\pi_{\mathfrak{m}}^\vee/\mathfrak{m}^{r+1}\pi_{\mathfrak{m}}^\vee.
\end{align*}
But $\pi_{\mathfrak{m}}^\vee$ is assumed to be flat over $\mathbb{T}(C^{\mathfrak{p}})_{\mathfrak{m}}$, we decuce that
\begin{align*}
    \mathfrak{m}^{r}\pi_{\mathfrak{m}}^\vee/\mathfrak{m}^{r+1}\pi_{\mathfrak{m}}^\vee & \cong \pi_{\mathfrak{m}}^\vee\otimes(\mathfrak{m}^r/\mathfrak{m}^{r+1}) \\
& \cong (\pi_{\mathfrak{m}}^\vee \otimes \mathbb{T}(C^{\mathfrak{p}})_{\mathfrak{m}}/\mathfrak{m})\otimes_{\mathbb{T}(C^{\mathfrak{p}})_{\mathfrak{m}}/\mathfrak{m}} (\mathfrak{m}^r/\mathfrak{m}^{r+1}) \\
& \cong (\pi_{\mathfrak{m}}[\mathfrak{m}])^\vee \otimes_{\mathbb{T}(C^{\mathfrak{p}})_{\mathfrak{m}}/\mathfrak{m}} (\mathfrak{m}^r/\mathfrak{m}^{r+1}).
\end{align*}
From the isomorphism $\pi[\mathfrak{m}] = \pi_{\mathfrak{m}}[\mathfrak{m}]$, the last term above can be identified with
$$(\pi_{\mathfrak{m}}[\mathfrak{m}])^\vee \otimes_{\mathbb{T}(C^{\mathfrak{p}})_{\mathfrak{m}}/\mathfrak{m}} (\mathfrak{m}^r/\mathfrak{m}^{r+1}) \cong (\pi[\mathfrak{m}])^\vee \otimes_{\mathbb{T}/\mathfrak{m}} (\mathfrak{m}^r/\mathfrak{m}^{r+1}) \cong \bigoplus_{s}(\pi[\mathfrak{m}])^\vee$$
where $s = \rm{dim}_{\mathbb{T}/\mathfrak{m}} (\mathfrak{m}^r/\mathfrak{m}^{r+1})$. Combining all these isomorphisms and taking duals again, we have $\pi_{\mathfrak{m}}[\mathfrak{m}^{r+1}]/\pi_{\mathfrak{m}}[\mathfrak{m}^r] \cong \big((\pi[\mathfrak{m}])^\vee \otimes_{\mathbb{T}/\mathfrak{m}} (\mathfrak{m}^r/\mathfrak{m}^{r+1})\big)^\vee \cong (\bigoplus_{s}(\pi[\mathfrak{m}])^\vee)^\vee \cong \bigoplus_s \pi[\mathfrak{m}]$ where $s = \rm{dim}_{\mathbb{T}/\mathfrak{m}} (\mathfrak{m}^r/\mathfrak{m}^{r+1})$, as desired.
\end{proof}

\begin{remark}\label{Gee-Newton}
In \cite{GN16}, Theorem B, Gee-Newton proved among other things that under certain assumptions on the Gelfand-Kirillov dimension of $\pi$, the corresponding flatness result for the group $\text{PGL}_n$ indeed holds, so our flatness assumption above seems to be reasonable and one might expect a proof of it under similar conditions as in \cite{GN16} or directly via the Jacquet-Langlands correspondence using Theorem B of \cite{GN16}.
\end{remark}

Assume from now on that $\pi_{\mathfrak{m}}^\vee$ is flat over $\mathbb{T}(C^{\mathfrak{p}})_{\mathfrak{m}}$. As a consequence of the above lemma, for each $r \geq 1$ the sheaf $\mathcal{F}_{{\pi_{\mathfrak{m}}[\mathfrak{m}^{r+1}]/\pi_{\mathfrak{m}}[\mathfrak{m}^r]}}$ and hence the cohomology group\\ $H^{n-1}_{\text{ét}}(\mathbb{P}_{\mathbb{C}_p}^{n-1}, \mathcal{F}_{{\pi_{\mathfrak{m}}[\mathfrak{m}^{r+1}]/\pi_{\mathfrak{m}}[\mathfrak{m}^r]}})$, are determined by the torsion part $\pi[\mathfrak{m}]$. We also note that

\begin{lemma}\label{powertorsion}
The $\mathbb{T}(C^{\mathfrak{p}})_{\mathfrak{m}}$-module $\pi_{\mathfrak{m}}$ is $\mathfrak{m}$ power torsion: $\pi_{\mathfrak{m}}=   \displaystyle{\varinjlim_{k}} \, \pi_\mathfrak{m}[\mathfrak{m}^k] $. Moreover, there exists an integer $N > 0$ such that $\bar{\sigma}|_{\mathrm{Gal}_{F_{\mathfrak{p}}}}$ appears in $H^{n-1}_{\text{ét}}(\mathbb{P}_{\mathbb{C}_p}^{n-1}, \mathcal{F}_{\pi_{\mathfrak{m}}[\mathfrak{m}^N]})[\mathfrak{m}]$.
\end{lemma}

\begin{proof}
By the equality $$\pi = \varinjlim_{K_\mathfrak{p}}C^0(G'(F)\backslash G'(\mathbb{A}_{F,f})/(K_\mathfrak{p}\times \mathcal{O}^\times_{F_\mathfrak{p}}  \times C^{\mathfrak{p}} ), \mathbb{Q}_p/\mathbb{Z}_p)$$ 
and the finiteness of the cardinality of $G'(F)\backslash G'(\mathbb{A}_{F,f})/(K_\mathfrak{p}\times \mathcal{O}^\times_{F_\mathfrak{p}}  \times C^{\mathfrak{p}} )$ it suffices to show that for each compact open subgroup $K_\mathfrak{p} \subseteq G(F_\mathfrak{p})$ and for each positive integer $r$, the localization at $\mathfrak{m}$ of $$\pi_{K_\mathfrak{p},r}:= C^0(G'(F)\backslash G'(\mathbb{A}_{F,f})/(K_\mathfrak{p}\times \mathcal{O}^\times_{F_\mathfrak{p}}  \times C^{\mathfrak{p}} ), \mathbb{Z}/p^r\mathbb{Z})$$ is $\mathfrak{m}$ power torsion. Let $\tilde{\mathbb{T}}$ be the image of the map $\phi: \mathbb{T} \to \text{End}(\pi_{K_\mathfrak{p}, r})$ and $\tilde{\mathfrak{m}} \subseteq \tilde{\mathbb{T}}$ be the image of $ \mathfrak{m}$ under $\phi$. We distinguish two cases:

\begin{itemize}
    \item $\tilde{\mathfrak{m}}$ is equal to $\tilde{\mathbb{T}}$, which is equivalent to $\text{ker}(\phi) \nsubseteq \mathfrak{m}$. In this case it is direct to check that $(\pi_{K_\mathfrak{p}, r})_\mathfrak{m} =0$.
    \item $\text{ker}(\phi) \subseteq \mathfrak{m}$, in which case $\tilde{\mathfrak{m}}$ is a maximal ideal of $\tilde{\mathbb{T}}$. As $\pi_{K_\mathfrak{p},r}$ is finite as a set it follows that $\tilde{\mathbb{T}}$ is a finite ring and hence an Artin ring. Thus $\tilde{\mathbb{T}}_{\tilde{\mathfrak{m}}}$ is a local Artin ring and there exists an integer $M > 0$ such that $\tilde{\mathfrak{m}}^M\tilde{\mathbb{T}}_{\tilde{\mathfrak{m}}}=0$. Now $\mathfrak{m}^M(\pi_{K_\mathfrak{p}, r})_\mathfrak{m} = \tilde{\mathfrak{m}}^M(\tilde{\mathbb{T}}_{\tilde{\mathfrak{m}}}\otimes_{\tilde{\mathbb{T}}}\pi_{K_\mathfrak{p}, r}) =0$.
\end{itemize}

For the second statement, we first note that from the exactness of $\pi \mapsto \mathcal{F}_\pi$ we have an injective morphism of sheaves
$$\varinjlim_{k}\mathcal{F}_{\pi_\mathfrak{m}[\mathfrak{m}^k]} \to \mathcal{F}_{\pi_\mathfrak{m}},$$ which is also surjective by checking stalks using the equality $\mathcal{F}_{\pi, \bar{x}} = \pi$ and the first statement of the lemma. On the other hand, by Proposition 2.8 of \cite{Sch18} and the coherence of the sites $(\mathbb{P}_{\mathbb{C}_p}^{n-1}/K)_{\text{\'et}}$ we deduce 
$$H^{n-1}_{\text{ét}}(\mathbb{P}_{\mathbb{C}_p}^{n-1}, \mathcal{F}_{\pi_{\mathfrak{m}}}) = \varinjlim_kH^{n-1}_{\text{ét}}(\mathbb{P}_{\mathbb{C}_p}^{n-1}, \mathcal{F}_{\pi_{\mathfrak{m}}[\mathfrak{m}^k]})$$
with injective transition maps in the direct system. As $\bar{\sigma}|_{\mathrm{Gal}_{F_{\mathfrak{p}}}}$ is finite dimensional and occurs in $H^{n-1}_{\text{ét}}(\mathbb{P}_{\mathbb{C}_p}^{n-1}, \mathcal{F}_{\pi_{\mathfrak{m}}})[\mathfrak{m}]=\bar{\sigma}|_{\mathrm{Gal}_{F_{\mathfrak{p}}}}\otimes\rho_{C^{\mathfrak{p}}}[\mathfrak{m}]$ it must occur in some $H^{n-1}_{\text{ét}}(\mathbb{P}_{\mathbb{C}_p}^{n-1}, \mathcal{F}_{\pi_{\mathfrak{m}}[\mathfrak{m}^N]})[\mathfrak{m}]$ for some large $N$.

\end{proof}

Now we can prove

\begin{theorem}\label{torsionthm1}
Assume that $\pi_{\mathfrak{m}}^\vee$ is flat over $\mathbb{T}(C^{\mathfrak{p}})_{\mathfrak{m}}$ and that $\bar{\sigma}|_{\mathrm{Gal}_{F_{\mathfrak{p}}}}$ is semisimple with distinct irreducible factors. Then $\mathrm{JH}(\bar{\sigma}|_{\mathrm{Gal}_{F_{\mathfrak{p}}}})$ is a subset of the set of irreducible subquotients $H^{n-1}_{\text{ét}}(\mathbb{P}_{\mathbb{C}_p}^{n-1}, \mathcal{F}_{\pi[\mathfrak{m}]})$.
\end{theorem}

\begin{proof}
We may assume that $\bar{\sigma}|_{\mathrm{Gal}_{F_{\mathfrak{p}}}}$ is irreducible.
Pick an $N >0$ such that $\bar{\sigma}|_{\mathrm{Gal}_{F_{\mathfrak{p}}}}$ appears in $H^{n-1}_{\text{ét}}(\mathbb{P}_{\mathbb{C}_p}^{n-1}, \mathcal{F}_{\pi_{\mathfrak{m}}[\mathfrak{m}^N]})[\mathfrak{m}]$ whose existence is guaranteed by the above lemma. Denote $\pi_N := \pi_{\mathfrak{m}}[\mathfrak{m}^N]$ and consider the short exact sequence
$$0 \to \pi_{N}[\mathfrak{m}] \overset{\tilde{i}}{\to} \pi_{N} \to \pi_{N}/ \pi_{N}[\mathfrak{m}] \to 0$$
and its associated long exact sequence 
$$ \ldots \to H^{n-1}_{\text{ét}}(\mathbb{P}_{\mathbb{C}_p}^{n-1}, \mathcal{F}_{\pi_{N}[\mathfrak{m}]}) \overset{\tilde{i}_*}{\to} H^{n-1}_{\text{ét}}(\mathbb{P}_{\mathbb{C}_p}^{n-1}, \mathcal{F}_{\pi_{N}}) \to H^{n-1}_{\text{ét}}(\mathbb{P}_{\mathbb{C}_p}^{n-1}, \mathcal{F}_{C_1}) \to \ldots$$
where $C_1:=\pi_{N}/ \pi_{N}[\mathfrak{m}]$. Since $\pi_{N}[\mathfrak{m}]$ is of $\mathfrak{m}$-torsion, by functoriality the map $\tilde{i}_*$ factors through $H^{n-1}_{\text{ét}}(\mathbb{P}_{\mathbb{C}_p}^{n-1}, \mathcal{F}_{\pi_{N}})[\mathfrak{m}]$ which by assumption contains $\bar{\sigma}|_{\mathrm{Gal}_{F_{\mathfrak{p}}}}$. If the image of $\tilde{i}_*$ contains a copy of $\bar{\sigma}|_{\mathrm{Gal}_{F_{\mathfrak{p}}}}$ in $H^{n-1}_{\text{ét}}(\mathbb{P}_{\mathbb{C}_p}^{n-1}, \mathcal{F}_{\pi_{N}})$ then we are done, as the irreducible representation $\bar{\sigma}|_{\mathrm{Gal}_{F_{\mathfrak{p}}}}$ occurs in the composition series of a quotient of $H^{n-1}_{\text{ét}}(\mathbb{P}_{\mathbb{C}_p}^{n-1}, \mathcal{F}_{\pi_{N}[\mathfrak{m}]})$ (note that $\pi_{N}[\mathfrak{m}] = \pi[\mathfrak{m]})$. Now suppose the contrary so that coker$(\tilde{i}_*)$, containing a copy of $\bar{\sigma}|_{\mathrm{Gal}_{F_{\mathfrak{p}}}}$ (since $\bar{\sigma}|_{\mathrm{Gal}_{F_{\mathfrak{p}}}}$ is assumed to be irreducible), injects into $H^{n-1}_{\text{ét}}(\mathbb{P}_{\mathbb{C}_p}^{n-1}, \mathcal{F}_{{C_1}})$. 

Consider the further filtration
$$0 \to C_1[\mathfrak{m}] \overset{i_1}{\to} C_1 \to C_2 \to 0$$
where $C_2:=C_1/C_1[\mathfrak{m}]$ and one can check that $C_1[\mathfrak{m}]=\pi_{N}[\mathfrak{m}^2]/\pi_{N}[\mathfrak{m}]$ so that $C_2=\pi_{N}/\pi_{N}[\mathfrak{m}^2]$. We have again the long exact sequence
$$ \ldots \to H^{n-1}_{\text{ét}}(\mathbb{P}_{\mathbb{C}_p}^{n-1}, \mathcal{F}_{C_1[\mathfrak{m}]}) \overset{i_{1*}}{\to}  H^{n-1}_{\text{ét}}(\mathbb{P}_{\mathbb{C}_p}^{n-1}, \mathcal{F}_{{C_1}}) \to H^{n-1}_{\text{ét}}(\mathbb{P}_{\mathbb{C}_p}^{n-1}, \mathcal{F}_{C_2}) \to \ldots$$
By the same argument as above, either the image of $i_{1*}$ contains a copy of $\bar{\sigma}|_{\mathrm{Gal}_{F_{\mathfrak{p}}}}$ inside $H^{n-1}_{\text{ét}}(\mathbb{P}_{\mathbb{C}_p}^{n-1}, \mathcal{F}_{{C_1}})$, or, otherwise, coker$(i_{1*})$ contains a copy of $\bar{\sigma}|_{\mathrm{Gal}_{F_{\mathfrak{p}}}}$ and injects into $H^{n-1}_{\text{ét}}(\mathbb{P}_{\mathbb{C}_p}^{n-1}, \mathcal{F}_{C_2})$. Then one defines $C_3:=\pi_{N}/\pi_{N}[\mathfrak{m}^3]$ to continue with the above procedure.

Considering that $\pi_{N}=\pi_{N}[\mathfrak{m}^N]$ (hence $C_N:=\pi_{N}/\pi_{N}[\mathfrak{m}^N]=0$), one concludes by induction that there exists a positive integer $r < N$ such that the image of $i_{r*}$ contains a copy of $\bar{\sigma}|_{\mathrm{Gal}_{F_{\mathfrak{p}}}}$ inside $H^{n-1}_{\text{ét}}(\mathbb{P}_{\mathbb{C}_p}^{n-1}, \mathcal{F}_{{C_r}})$, i.e., $\bar{\sigma}|_{\mathrm{Gal}_{F_{\mathfrak{p}}}}$ is an irreducible sub-representation of a homomorphic image of $H^{n-1}_{\text{ét}}(\mathbb{P}_{\mathbb{C}_p}^{n-1}, \mathcal{F}_{{C_r[\mathfrak{m}]}})$ and therefore appears in the composition series of $$H^{n-1}_{\text{ét}}(\mathbb{P}_{\mathbb{C}_p}^{n-1}, \mathcal{F}_{{C_r[\mathfrak{m}]}}) = H^{n-1}_{\text{ét}}(\mathbb{P}_{\mathbb{C}_p}^{n-1}, \mathcal{F}_{{\pi_{\mathfrak{m}}[\mathfrak{m}^{r+1}]/\pi_{\mathfrak{m}}[\mathfrak{m}^r]}})$$ where we have used the equalities $\pi_N[\mathfrak{m}^i]=\pi[\mathfrak{m}^i]$ for $1 \leq i \leq N$.

So we have proved that $\bar{\sigma}|_{\mathrm{Gal}_{F_{\mathfrak{p}}}}$ appears in the composition series either of\\ $H^{n-1}_{\text{ét}}(\mathbb{P}_{\mathbb{C}_p}^{n-1}, \mathcal{F}_{\pi[\mathfrak{m}]})$ or of $H^{n-1}_{\text{ét}}(\mathbb{P}_{\mathbb{C}_p}^{n-1}, \mathcal{F}_{{\pi_{\mathfrak{m}}[\mathfrak{m}^{r+1}]/\pi_{\mathfrak{m}}[\mathfrak{m}^r]}})$ for some positive integer $r$, but in the latter case one has
$$H^{n-1}_{\text{ét}}(\mathbb{P}_{\mathbb{C}_p}^{n-1}, \mathcal{F}_{{\pi_{\mathfrak{m}}[\mathfrak{m}^{r+1}]/\pi_{\mathfrak{m}}[\mathfrak{m}^r]}}) \cong \bigoplus_s H^{n-1}_{\text{ét}}(\mathbb{P}_{\mathbb{C}_p}^{n-1}, \mathcal{F}_{ \pi[\mathfrak{m}]})$$
from the equality $\pi_{\mathfrak{m}}[\mathfrak{m}^{r+1}]/\pi_{\mathfrak{m}}[\mathfrak{m}^r] \cong \bigoplus_s \pi[\mathfrak{m}]$ (hence $\mathcal{F}_{{\pi_{\mathfrak{m}}[\mathfrak{m}^{r+1}]/\pi_{\mathfrak{m}}[\mathfrak{m}^r]}} \cong \bigoplus_s \mathcal{F}_{ \pi[\mathfrak{m}]}$).

\end{proof}

Now we show that when $\bar{\sigma}|_{\mathrm{Gal}_{F_{\mathfrak{p}}}}$ is irreducible, it is uniquely determined by $\pi[\mathfrak{m}]$ in the sense that we can read off $\bar{\sigma}|_{\mathrm{Gal}_{F_{\mathfrak{p}}}}$ from $H^{n-1}_{\text{ét}}(\mathbb{P}_{\mathbb{C}_p}^{n-1}, \mathcal{F}_{ \pi[\mathfrak{m}]})$.

\begin{lemma}\label{injective}
Assume that $\pi_{\mathfrak{m}}^\vee$ is flat over $\mathbb{T}(C^{\mathfrak{p}})_{\mathfrak{m}}$. Then $H^{i}_{\text{ét}}(\mathbb{P}_{\mathbb{C}_p}^{n-1}, \mathcal{F}_{\pi[\mathfrak{m}]})=0$ for $1 \leq i \leq n-2$, and the map $H^{n-1}_{\text{ét}}(\mathbb{P}_{\mathbb{C}_p}^{n-1}, \mathcal{F}_{\pi_{\mathfrak{m}}[\mathfrak{m}]}) \overset{i_*}{\to} H^{n-1}_{\text{ét}}(\mathbb{P}_{\mathbb{C}_p}^{n-1}, \mathcal{F}_{\pi_{\mathfrak{m}}})[\mathfrak{m}]$ is injective.
\end{lemma}

 \begin{proof}
We prove that the kernel of the map $i_*$ in the long exact sequence \eqref{longes1}, which is equal to the image of the map $$ H^{n-2}_{\text{ét}}(\mathbb{P}_{\mathbb{C}_p}^{n-1}, \mathcal{F}_{C}) \to H^{n-1}_{\text{ét}}(\mathbb{P}_{\mathbb{C}_p}^{n-1}, \mathcal{F}_{\pi_{\mathfrak{m}}[\mathfrak{m}]}),$$ is trivial by showing directly that $ H^{n-2}_{\text{ét}}(\mathbb{P}_{\mathbb{C}_p}^{n-1}, \mathcal{F}_{C})=0$. This will be a consequence of the hypothesis that $H^{i}(\mathrm{Sh}_{UC^\mathfrak{p},\mathbb{C}},\mathbb{Z}_p)_\mathfrak{m}$ is concentrated in middle degree. Indeed, in degree $0$ this follows from the injection
$$ H^{0}_{\text{ét}}(\mathbb{P}_{\mathbb{C}_p}^{n-1}, \mathcal{F}_{C}) \to \bigoplus_{i=1}^{r} H^{0}_{\text{ét}}(\mathbb{P}_{\mathbb{C}_p}^{n-1}, \mathcal{F}_{\pi_{\mathfrak{m}}}) = \bigoplus_{i=1}^{r}H^0(\mathrm{Sh}_{C^{\mathfrak{p}}, \mathbb{C}_p}, \mathbb{Q}_p/\mathbb{Z}_p)_\mathfrak{m} $$
induced by the exact sequence \eqref{shortes2}, and from the vanishing of $H^0(\mathrm{Sh}_{C^{\mathfrak{p}}, \mathbb{C}_p}, \mathbb{Q}_p/\mathbb{Z}_p)_\mathfrak{m}$ by our assumption on $\mathfrak{m}$.
Then we proceed by induction, so assume now that $H^{k}_{\text{ét}}(\mathbb{P}_{\mathbb{C}_p}^{n-1}, \mathcal{F}_{C})=0$ for some $k$ with $0 \leq k \leq n-3$ and we want to prove the vanishing in degree $k+1$. Note that from the sequence \eqref{shortes1} we have immediately $H^{k+1}_{\text{ét}}(\mathbb{P}_{\mathbb{C}_p}^{n-1}, \mathcal{F}_{\pi[\mathfrak{m}]})=0$.
Further, as we can write $\pi_\mathfrak{m}/\pi_\mathfrak{m}[\mathfrak{m}]=\displaystyle{\varinjlim_{N}} \, \pi_\mathfrak{m}[\mathfrak{m}^N]/\pi_\mathfrak{m}[\mathfrak{m}]$ by Lemma \ref{powertorsion}, it suffices to prove $H^{k+1}_{\text{ét}}(\mathbb{P}_{\mathbb{C}_p}^{n-1}, \mathcal{F}_{C'_N})=0$ for all $N\geq 1$ where $C'_N :=\pi_\mathfrak{m}[\mathfrak{m}^N]/\pi_\mathfrak{m}[\mathfrak{m}]$. But this follows directly by the induction process in the proof of Theorem \ref{torsionthm1}, which shows in particular that $C'_N$ is a successive extension of copies of $\pi[\mathfrak{m}]$ (by Lemma \ref{flatness}).

 \end{proof}

\begin{theorem}
Assume that $\pi_{\mathfrak{m}}^\vee$ is flat over $\mathbb{T}(C^{\mathfrak{p}})_{\mathfrak{m}}$. Then $H^{n-1}_{\text{ét}}(\mathbb{P}_{\mathbb{C}_p}^{n-1}, \mathcal{F}_{ \pi[\mathfrak{m}]})$ is a non-zero admissible $\mathrm{Gal}_{F_{\mathfrak{p}}} \times B_\mathfrak{q}^\times $-representation, and any of its indecomposable  $\mathrm{Gal}_{F_{\mathfrak{p}}}$-subrepresentations is a subrepresentation of $\bar{\sigma}|_{\mathrm{Gal}_{F_{\mathfrak{p}}}}$. In particular, if $\bar{\sigma}|_{\mathrm{Gal}_{F_{\mathfrak{p}}}}$ is irreducible, then every indecomposable subrepresentation of $H^{n-1}_{\text{ét}}(\mathbb{P}_{\mathbb{C}_p}^{n-1}, \mathcal{F}_{ \pi[\mathfrak{m}]})$ is isomorphic to $\bar{\sigma}|_{\mathrm{Gal}_{F_{\mathfrak{p}}}}$. 
\end{theorem}

\begin{proof}
These are direct consequences of Lemma \ref{injective} and the proof of Theorem \ref{torsionthm1}.
\end{proof}

\begin{remark}\label{rmkflatness}
In the $n=2$ case of \cite{Sch18}, no flatness condition on $\pi_{\mathfrak{m}}^\vee$ is needed, whereas it plays an essential part in obtaining our results. This difference can be explained as follows. When $n=2$, the comparison map
$$H^{n-1}_{\text{ét}}(\mathbb{P}_{\mathbb{C}_p}^{n-1}, \mathcal{F}_{\pi_{\mathfrak{m}}[\mathfrak{m}]}) \overset{i_*}{\to} H^{n-1}_{\text{ét}}(\mathbb{P}_{\mathbb{C}_p}^{n-1}, \mathcal{F}_{\pi_{\mathfrak{m}}})[\mathfrak{m}]$$
is shown to have finite dimensional cokernel, which (together with the triviality of the kernel of $i_*$) suffices to give a complete classification result on possible Galois module structures of $H^{n-1}_{\text{ét}}(\mathbb{P}_{\mathbb{C}_p}^{n-1}, \mathcal{F}_{ \pi[\mathfrak{m}]})$ in terms of $\bar{\sigma}|_{\mathrm{Gal}_{F_{\mathfrak{p}}}}$, cf. Theorem 7.8 and its proof in \cite{Sch18}. The finite dimensionality of $\mathrm{coker}(i_*)$ is in turn a consequence of Strauch's description of the action of $\mathrm{GL}_n(F_\mathfrak{p})\times W_{F_\mathfrak{p}} \times B^\times_\mathfrak{q}$ on the geometric connected components $\pi_0(\mathcal{M}_{\mathrm{LT},\infty,\mathbb{C}_p}) \cong F_\mathfrak{p}^\times$; see Proposition 4.7 and its proof in \cite{Sch18} (in fact we only need the result that $B^\times_\mathfrak{q}$ acts via the reduced norm map to $F_\mathfrak{p}^\times$). 

When $n > 2$, the cokernel of $i_*$ is given in terms of higher cohomology groups rather than $H^0$, as exhibited in the sequences (\ref{longes1}) and (\ref{longes2}), for which we have currently no similar descriptions of the group action of $B^\times_\mathfrak{q}$; furthermore, the subgroup $(\mathcal{O}^\times_{B_\mathfrak{q}})_1 \subseteq \mathcal{O}^\times_{B_\mathfrak{q}}$ of elements of reduced norm 1 may act non-trivially on the cokernel in these cases.  Instead our flatness assumption enables us to relate the graded pieces in the filtration $\ldots \subseteq \pi_\mathfrak{m}[\mathfrak{m}^r] \subseteq \pi_\mathfrak{m}[\mathfrak{m}^{r+1}] \subseteq \ldots$ of $\pi_\mathfrak{m}$ and $\pi[\mathfrak{m}]$, and then obtain partial information on $H^{n-1}_{\text{ét}}(\mathbb{P}_{\mathbb{C}_p}^{n-1}, \mathcal{F}_{ \pi[\mathfrak{m}]})$ in certain cases. However it is expected that such flatness condition is more of technical nature rather than conceptual.

\end{remark}

\begin{remark}
In \cite{Sch18}, Scholze also proved a compatibility result of his functor with the patching construction of \cite{CEG$^+$18}, again for $n=2$. It is expected such compatibility result holds for $n>2$ too. To prove it one may adapt Scholze's argument to the patching context for compact unitary groups as in \cite{BLGG11}, noticing that the main result in Section 8 of \cite{Sch18} works in arbitrary dimensions to control the behaviour of the functor $H^{n-1}_{\text{ét}}(\mathbb{P}_{\mathbb{C}_p}^{n-1}, \mathcal{F}_\bullet) $ with respect to taking limits.
\end{remark}

\medskip

\end{section}

\end{document}